\documentclass[12pt]{amsart}

\usepackage{comment,hyperref,color, epigraph}

\usepackage{breakurl}   %this allows for url breaks in pdf

\usepackage{enumitem} %this allows proper labeling in \enumerate

\usepackage{comment} %for longer comments

\usepackage[dvips]{graphicx} 

\usepackage{setspace}  %for double spacing

\numberwithin{equation}{section}

\newtheorem{theorem}{Theorem}[section]

\newtheorem{proposition}[theorem]{Proposition}
\newtheorem{corollary}[theorem]{Corollary}
\newtheorem{definition}[theorem]{Definition}

\newcommand{\N}{\mathbb{N}}

\newcommand{\Q}{\mathbb{Q}}
\newcommand{\R}{\mathbb{R}}

\newcommand\astr{{{}^\ast\hspace{-.5pt}\R}}
\newcommand\astn{{{}^\ast\hspace{-.5pt}\N}}
\newcommand\asta{{{}^\ast\hspace*{-3.5pt}A}}
\newcommand{\astb}{{}^{\ast}\hspace*{-2.5pt}B}

\newcommand{\ns}[1]{{}^\ast\hspace{-1pt}#1}
\newcommand\st{{\rm st}}

\newcommand{\sy}{\textbf{Sy}}

\newcommand{\co}{\textbf{Co}}

\newcommand{\Un}{\textbf{Un}}
\newcommand{\Po}{\textbf{Po}}

\renewcommand*{\MR}[1]{\href{http://www.ams.org/mathscinet-getitem?mr=#1&return=pdf}{MR #1}}
\newcommand*{\ZBL}[1]{\href{http://www.zentralblatt-math.org/zmath/en/advanced/?q=an:#1&format=complete}{Zbl #1}}

\title[History of Archimedean and non-Archimedean approaches]
{History of Archimedean and non-Archimedean approaches to uniform
  processes: Uniformity, symmetry, regularity}

\author{Emanuele Bottazzi} \address{Department of Mathematics,
  University of Pavia, Via Adolfo Ferrata 5, 27100 Pavia, Italy}
\email{emanuele.bottazzi@unipv.it, emanuele.bottazzi.phd@gmail.com}

\author{Mikhail G. Katz} \address{Department of Mathematics, Bar Ilan
  University, Ramat Gan 5290002 Israel} \email{katzmik@math.biu.ac.il}

\date{\today}

\begin{document}

%\doublespacing

\begin{abstract}
We apply Nancy Cartwright's distinction between theories and basic
models to explore the history of rival approaches to modeling a notion
of chance for an ideal uniform physical process known as a fair
spinner.  This process admits both Archimedean and non-Archimedean
models.  Advocates of Archimedean models maintain that the fair
spinner should satisfy hypotheses such as invariance with respect to
rotations by an arbitrary real angle, and assume that the optimal
mathematical tool in this context is the Lebesgue measure.  Others
argue that invariance with respect to all real rotations does not
constitute an essential feature of the underlying physical process,
and could be relaxed in favor of regularity.  We show that, working in
ZFC, no subset of the commonly assumed hypotheses determines a unique
model, suggesting that physically based intuitions alone are
insufficient to pin down a unique mathematical model.  We provide a
rebuttal of recent criticisms of non-Archimedean models by Parker and
Pruss.

\end{abstract}


\maketitle
\tableofcontents

%\epigraph{His notions fitted things so well,\\ That which was which he
%  could not tell;\\ But oftentimes mistook th' one For th' other,
%  {\ldots} --Samuel Butler, \emph{Hudibras}}

\section{Introduction}

We explore the history of Archimedean and non-Archi\-medean models for
uniform chances.  We compare the properties of Archimedean and
non-Archi\-medean models for uniform chances and review rival
approaches to modeling uniform chances in Section \ref{f3b}.  We focus
on the hypotheses which underlie such rival formalisations of an ideal
uniform physical process known as a fair spinner in Section
\ref{section spinner}.

Advocates of Archimedean models maintain that the fair spinner should
satisfy hypotheses such as invariance with respect to rotations by an
arbitrary real angle,
%uniformity and assigning a chance to every point, 
and assume that the optimal mathematical tool in this context is the
Lebesgue measure.  Thus, Pruss writes:
\begin{enumerate}\item[]
[I]t seems that regardless of whether we prefer full conditional,
regular hyperreal or regular qualitative probabilities, we need to
abandon our symmetry intuitions in some but not other cases in ways
that seem intuitively ad hoc.  It may thus be intuitively better to
stick to the classical system of real-valued probabilities, where not
all subsets have defined probabilities, where we lack regularity, but
it is much easier to get symmetries.%
\footnote{Pruss \cite[p.\;8525]{Pr21b}.}
\end{enumerate}
He concludes:
\begin{enumerate}\item[]
\emph{Lebesgue measure} on Euclidean space is always translation- and
rotation-invariant, and product measures for infinite coin-flip
situations will be invariant under all shifts, and indeed under all
permutations of the coins.%
\footnote{Ibid.; emphasis added.}
\end{enumerate}
Pruss goes on to claim that ``the classical approach is superior in
allowing for the symmetries.''%
\footnote{Ibid.  In a separate text \cite{Pr21a}, Pruss claims to show
that a non-Archimedean probability function~$P$ can always be used to
define another non-Archimedean probability function~$P'$ such that
$P'$ preserves all the decision-theoretic verdicts delivered by
expected utility reasoning calculated using~$P$.  Bottazzi and Katz
\cite{21b} demonstrated that such a claim has no merit, for the
following reason.  Pruss's functions~$P'$ are obtained from~$P$ by
rescaling the infinitesimal part of~$P$; but all such~$P'$ are
\emph{external} entities, as shown in \cite{21b}.  Meanwhile, only
\emph{internal} entities are relevant to the practice of nonstandard
analysis, since they are the entities to which the transfer principle
applies; see further in Section~\ref{f3}.}

We consider both Archimedean (Section \ref{sec A-track}) and
non-Archi\-medean models for the spinner (Section \ref{sec B-track})
and discuss how they are positioned with respect to the hypotheses
discussed in Section \ref{section spinner}.  We show that, working in
Zermelo--Fraenkel set theory with the Axiom of Choice (ZFC), no subset
of these hypotheses determines a unique model, whether Archimedean or
non-Archi\-medean.  Thus we argue that physically based intuitions
alone do not enable one to pin down a unique mathematical model.  In
Section \ref{s55} we analyze Parker's critique based on haecceitism
targeting nonstandard chances.  We review the necessary material on
Robinson's framework in Section~\ref{s6}.  We present our conclusions
in Section~\ref{s7}.

\section{History of regular chances}
\label{f3b}

In the foundations of probability theory, the role of events of zero
chance%
\footnote{For a survey on the relations between the notions of chance
and credence, see e.g., Spohn~\cite{spohn}.  The distinction goes back
to Carnap.}
has a rich history.
%
% On finite sample spaces, events of zero chance are impossible,
%namely no rational agent should believe in their occurrence. This
%position stems both from the classical definition of probability as
%the ratio of favourable outcomes over all the possible outcomes and
%from the subjectivist definition associated mainly with De Finetti.
%
Such events typically arise when considering models involving
real-valued probabilities on infinite sample spaces.  Under natural
hypotheses, Archimedean probabilities (namely, probability measures
that take values in Archimedean ordered fields, usually the field of
real numbers~$\R$) are forced to assign zero chance to some events
that can nevertheless occur.  Some relevant examples are models of
vast pluralities of trials/tickets, such as an infinite collection of
coin tosses or an infinite fair lottery.  Large market economies,%
\footnote{See e.g., Anderson et al.~\cite{An24}.}
random number generation in computer simulation, and possibly
predator/prey evolutions in an ecosystem are additional examples of
pre-mathematical phenomena involving a vast sample space, that are
typically studied using mathematical infinity of one kind or another.%

While the scientific hypothesis of a finite universe is commonly
accepted, even according to theories that reject such a hypothesis and
envision, say, an infinity of galaxies, it is difficult to incorporate
such infinity into instances of vast phenomena that actually merit
mathematical modeling, such as the phenomena mentioned above, for at
least two reasons:
\begin{enumerate}
\item
such phenomena are by their nature limited in time and space (it seems
impractical to envision a market economy spanning an infinity of
galaxies); and
\item
even admitting such infinite phenomena in the physical realm, there is
no guarantee that the essential mathematical feature of~$\N$, namely
the existence of a Markov shift~$n\mapsto n+1$, can be assumed to hold
there.%
\footnote{See further in note~\ref{f22}.}
\end{enumerate}

Archimedean theories of such gedankenexperiments are typically
obtained from real-valued probability measures (either
finitely-additive or~$\sigma$-additive) that assign a null measure to
each singleton.  In these cases, a typical approach consists in
distinguishing the notion of an impossible event from that of an event
of zero chance, the latter being a coarser category that includes the
former as well as other events.

However, these situations also admit models that assign a positive
chance to each outcome that is considered possible.  This property is
called regularity:
\begin{itemize}
\item[\textbf{Reg}] Every possible outcome has a positive chance.
\end{itemize}

A number of early advocates of regularity (or \emph{strict coherence})
are mentioned in H\'ajek~\cite{staying} and
Wenmackers~\cite[pp.\;244--246]{We19}.  Such early advocates include
Carnap~\cite{Ca63} (1963), Kemeny~\cite{Ke55} (1955), Edwards et
al.~\cite{Ed63} (1963), Shimony~(\cite{Sh55}, 1955 and \cite{Sh70},
1970), and Stalnaker~\cite{St70} (1970).
One of the advantages of
regularity is that it enables one to work with conditional
probabilities relative to an event of infinitesimal chance.  See e.g.,
\cite[Section 3.2]{benci2018}, \cite[Section 3.4]{Bo21},
\cite[Section~2.2]{21c} and references therein.

Real-valued chances on infinite sample spaces assigning equal
probability to singleton sets cannot be regular.  Therefore regular
models necessarily take values in a proper (and hence
non-Archi\-medean) extension of the field of real numbers.  Typically,
such regular models use a particular kind of extension of the real
numbers, such as fields of hyperreal numbers of Robinson's framework
\cite{Ro66} for analysis with infinitesimals.  There are other
non-Archi\-medean ordered extensions of the real numbers, but they are
currently not powerful enough to enable the development of models of
regular chances.%
\footnote{Bottazzi and Katz \cite[Section\;3.4]{21b},
\cite[Section\;4]{21c}.  For a recent approach that enables
non-hyperreal non-Archimedean chances in a \emph{partially} ordered
field extension of the real numbers, see Bottazzi and Eskew
\cite{Bo21}.}
%

%and Section~\ref{sec filterintegral}.

\subsection{History of critiques of regular chances}

Regular chances have been the target of repeated criticism in the
recent literature starting with Williamson (\cite{Wi07}, 2007).
Barrett~\cite{barrett10}, Parker~\cite{Pa21}, Pruss~\cite{Pr21a},
discuss mathematical models of chances on infinite sample spaces, such
as an infinite fair lottery and an infinite collection of coin tosses,
and claim that regular infinitesimal models possess undesirable
properties.

Parker focuses on an alleged `unattractiveness of regular chances',%
\footnote{Parker \cite[p.\;4]{Pa21}.}
claiming that in such models, ``[P]erfectly symmetric events under
perfectly symmetric circumstances must be assigned different chances
and credences.''%
\footnote{Ibid.}
Pruss claims that infinitesimal regular probabilities are
underdetermined, namely, that they ``carry more information than is
determined by the plausible kinds of constraints on these
probabilities.''%
\footnote{Pruss \cite[p.\;778]{Pr21a}.}
Such critiques have sometimes been anchored in purported `physical'
principles.  For instance, according to Parker,
\begin{enumerate}\item[]
Williamson (2007)%
\footnote{The reference is to Williamson \cite{Wi07}.}
argued that we should not require probabilities to be regular, for if
we do, certain `isomorphic' \emph{physical events} (infinite sequences
of coin flip outcomes) must have different probabilities, which is
implausible.%
\footnote{\label{f12}Parker \cite[Abstract]{Pa21}; emphasis added.}
\end{enumerate}
Parker evidently agrees with Williamson.  We analyze Parker's comment
in Section~\ref{f24}.

We will examine some of the physical principles mentioned in relation
to models of uniform physical processes.  Our analysis is occasioned
by the contribution by Parker \cite{Pa21} to the debate on the
properties of Archimedean and non-Archi\-medean models for uniform
chances.%
\footnote{See e.g., \cite{barrett10, benci2013, benci2018, 21b, 21c,
  Pr21a, Wi07, No22}.}
This topic is part of a broader discussion with regard to the
applicability of mathematics in the natural sciences.%
\footnote{Some recent articles include \cite{19c}, \cite{et},
  \cite{18i}.}

In Section~\ref{s55}, we address Parker's critique of regular
probabilities in terms of haecceitism.

\subsection{Critiques of regular probabilities}
\label{f24}

Parker and others have made several distinct criticisms of regular
probabilities, including
\begin{enumerate}
\item[(a)]
  the implausibility of assigning unequal
  probabilities to (allegedly) isomorphic events, and
  \item[(b)] excess
    structure, etc.
\end{enumerate}
We will address item (b) in Section~\ref{s55}.  Here we will focus on
item (a).

While criticizing regular probabilities, Parker speaks of ``certain
`isomorphic' physical events'' and mentions ``infinite sequences of
coin flip outcomes'' as an example (see above).

However, an \emph{infinite sequence} of coin flips is a basic
mathematical model%
\footnote{\label{f16}In mathematical modeling, we distinguish basic
mathematical models as opposed to full-fledged mathematical theories
brought to bear upon the analysis of the basic mathematical
models. For instance, the choice of~$\N$ in modeling a vast amount of
coin tosses is a basic mathematical model, whereas measure-theoretic
notions such as a choice of the sample space and a suitable measure
are not a part of such a basic mathematical model, but rather advanced
mathematical tools.  See further in Section~\ref{cart}.}
rather than the physical event itself.  If one takes the index set of
the flips to be~$\N$, then a subsequence obtained by a Markov shift
will indeed be isomorphic to the original sequence.  On the other
hand, if one takes the index set to be a hyperfinite set~$H$ (see
Section~\ref{f3}), any (internal) proper subset will be of strictly
smaller size than~$H$, in accordance with Euclid's part-whole
principle.  The latter is not `implausible,' contrary to Parker's
claim.  Parker's implausibility claim hinges on a conflation of the
realm of the phenomena being modeled and the realm of a basic
mathematical model.  In other words, Parker's preferred mathematical
framework (namely, the Archimedean one) influences his choice of the
preferred model (namely,~$\N$) for describing certain probabilistic
phenomena.

Similarly, Williamson (the originator of this type of argument)
describes the gedankenexperiment as follows:
\begin{quote}
[A] fair coin will be tossed infinitely many times at one second
intervals.%
\footnote{Williamson \cite[p.\,174]{Wi07}.  Note that
Williamson's ``infinitely many times at one second intervals'' cannot
be understood in terms of a hyperfinite set, since the natural numbers
form an external subset of the hyperfinite set~$\{1,2,3,\ldots,H\}$;
see further in Section~\ref{f3}.}
\end{quote}
Here the notion of a vast plurality of coin tosses is
pre-mathematical, while the structure implied by Williamson's
``infinitely many times at one second intervals''
%(strictly unnecessary for the mathematical description of this vast
%plurality of trials) 
is implicitly or explicitly assumed to be labeled by the set of
natural numbers~$\N$.%
\footnote{\label{f22}As mentioned in Section~\ref{f3b}, it is not
obvious that the Markov shift~$n\mapsto{n+1}$ naturally present
on~$\N$ could be assumed to be present in the physical realm even in
the hypothesis of physical theories that enable physical infinity.
Such a shift is necessary for the Williamson--Parker argument, as in
the claim that ``Regularity \ldots\ implies that the probability that
every coin toss in a given infinite sequence yields heads is smaller
than the probability that every toss after the first yields heads.
But these two events may be `isomorphic', etc.''  \cite[p.\;2]{Pa21}.
Without the Markov shift, there is no basis for the claim that these
events would be isomorphic.}
Even more explicitly, Parker goes on to describe Williamson's
gedankenexperiment in the following terms:
\begin{enumerate}\item[]
Williamson (2007) considers a \emph{countably infinite sequence} of
independent coin flips at one second intervals, with a fair coin.%
\footnote{Parker \cite[p.\;4]{Pa21}; emphasis added.}
\end{enumerate}
Parker seeks to describe the pre-mathematical concept of a vast
plurality of coin tosses, but ends up describing it in terms of its
basic mathematical model of a ``countably infinite sequence'' that, by
definition, is indexed by the formal mathematical entity~$\N$.
%
%Parker's notion (of~$\N$) fitted things (i.e., vast pluralities) so
%well that which was which he could not tell, and even mistook the one
%for the other (to paraphrase Samuel Butler).
%
Parker claims further that
\begin{enumerate}\item[]
[U]nder the standard Kolmogorov axioms, if the tosses are fair and
independent, the probability of an infinite sequence of heads must be
smaller than any positive real number.%
\footnote{Parker \cite[p.\;2]{Pa21}}
\end{enumerate}
However, the Kolmogorov axioms are, as the name suggests, already a
mathematical formalisation, and the ``infinite sequence'' in question
is clearly already indexed by~$\N$.  Therefore whatever conclusions
Parker draws from Kolmogorov's axiomatisation apply not to the realm
involving a pre-mathematical vast plurality of tosses, but rather to
the realm involving a basic mathematical model.  Accordingly, these
authors leave little room for the possibility of alternative
mathematical representations.  Their selection of the underlying
principles of the gedankenexperiment risks predetermining the outcome
of their arguments.
% see further in Section~\ref{s24}.

Concern about the risk of conflating the physical process with a basic
mathematical model was also voiced by Rainer Sachs in the context of
modeling of both quantum phenomena and general relativity.  Sachs and
Wu mention both the `jumpy' nature of the quantum world and the
apparent incompleteness of space-time.%
\footnote{``Whether incompleteness is a property of nature or a
  misleading feature of current models is a highly controversial
  question.  We remark that infinite extendibility of spacelike
  geodesics has no direct physical interpretation'' (Sachs and Wu
  \cite[p.\;28]{Sa77}).}
They point out that mathematicians -- due to their professional
training in differentiable manifolds -- tend to expect that nature
would `naturally' be modeled by smooth manifolds and tend to assume
that such models would have the formal property of completeness which
is `natural' by differential-geometric standards.  This is a case of
projecting one's mathematical training onto nature, similar to the
expectation that the pre-formal notion of a vast plurality is
`naturally' modeled by the formal entity~$\N$, causing philosophical
prejudice against other formal models.

The issue here is distinct from the philosophical issue regarding the
risk of ``taking the map for the territory'' in mathematical modeling
of physical phenomena.  Rather, the issue concerns the fact that the
\emph{same} physical situation may be usefully described by
\emph{distinct} basic mathematical models; see futher in
Section~\ref{cart}.

Critics of non-Archimedean probability acknowledge that they seek to
model physical phenomena rather than merely exploring \emph{logical
possibilities}.  Thus, Easwa\-ran's explicit premise is that ``All
\emph{physical} quantities can be entirely parametrized using the
standard real numbers.''%
\footnote{Easwaran \cite[Section~8.4, Premise~3]{Ea14}; emphasis
added.}
Regardless of whether the premise is justified,
%(and \cite{14a} argued that it is circular), 
it reveals that Easwaran's goal is ultimately to model \emph{physical}
phenomena.  Parker's text~\cite{Pa21} employs the adjective `physical'
over twenty times.  One of the earliest critics of non-Archimedean
chances, Adam Elga, published an article entitled ``Infinitesimal
chances and the laws of nature'' (indicating that he is dealing with
applicability of infinitesimals to modeling natural laws, rather than
merely with logical possibilities), where he mentions that
\begin{enumerate}\item[]
when continuously infinite state-spaces are involved, a great many
candidate systems (including systems of chancy laws that
\emph{physicists} have taken seriously) will ascribe zero chance to
any individual history%
\footnote{Elga \cite[p.\;69]{El04}; emphasis added.}
\end{enumerate}
and asks: ``infinitesimals to the rescue?''%
\footnote{Ibid.}

Another physical process that tends to be conflated with its
Archimedean model is one that Barrett \cite{barrett10} and
Pruss~\cite{Pr21a} refer to as a fair spinner, analyzed in
Section~\ref{section spinner}.

%A response to Pruss' underdetermination charges to infinitesimal
%probabilities \cite{Pr21a} appeared in \cite{21b,21c}.  In
%\cite{21c} Bottazzi and Katz also explicitly constructed some
%non-Archi\-medean models for the outcome of another uniform ideal
%physical process that Pruss as well as Barrett~\cite{barrett10} refer
%to as a fair spinner.

\subsection{History of the three realms}
\label{cart}

In 1995, Nancy Cartwright introduced a distinction between theories
and basic models when applying mathematical theories to analyze
phenomena in science (broadly conceived), in the context of the
history of the evolution of our understanding of the phenomenon of
superconductivity.  Cartwright emphasizes the role of basic
mathematical models (such as the London--London model of
superconductivity of 1935), to which a more elaborate mathematical
theory (such as the microscopic BCS theory of 1957) is applied.  Based
on her interpretation of the history of superconductivity, she argues
that one first builds a basic mathematical model of the phenomenon in
question, in a step Cartwright et al.~refer to as
\emph{phenomenological model-building}.%
\footnote{Cartwright et al.~\cite[p.\;148]{Ca95}; see also Su\'arez
  and Cartwright \cite[p.\;66]{Ca08} and references therein.}
%
%2022/10/26
Only then does a full-fledged mathematical theory emerge, providing an
account for the basic mathematical model.%
\footnote{Additional examples of scientific models are: ``The billiard
  ball model of the gas, the Bohr model of the atom, the
  Lotka--Volterra model of predator-prey interaction, general
  circulation models of the atmosphere, and agent-based models of
  social systems'' Frigg \cite[p.\,1]{Fr23}.}

Accordingly, in reference to probabilistic phenomena, one could
distinguish three realms:
\begin{enumerate}
	\item[(i)] pre-mathematical (more precisely, pre-formal)
          concepts best described as physical and/or
          naive-probabilistic phenomena, such as vast pluralities of
          trials, uniform processes, etc.;
	\item[(ii)] basic mathematical models thereof resulting from
          phenomenological model-building;
	\item[(iii)] advanced mathematical tools brought to bear upon
          the analysis of the basic mathematical models, such as the
          choice of a sample space and suitable measures such as the
          Lebesgue measure, hyperfinite counting measures, etc.
\end{enumerate}

A pre-mathematical phenomenon typically admits multiple basic
mathematical models, that in turn can be analyzed with distinct
advanced theories/tools.%
%2022/10/26
\footnote{Alternatives to classical logic can also be envisioned as
  tools in mathematical modeling: see e.g., Bridges and
  McKubre-Jordens \cite{failing}.  Among tools based on classical
  logic, infinitesimal approaches provide further alternatives to the
  classical Archimedean ones.}
Such non-uniqueness is not necessarily a drawback, since distinct
models could account for distinct aspects of the pre-mathematical
phenomenon of realm (i) that cannot be simultaneously described by a
single model.%
\footnote{A related point was made by Wenmackers
  \cite[p.\;239]{We19}.}

Whichever model one uses, it is important to keep distinct the
pre-mathematical concepts and the basic mathematical models, namely
realms (i) and (ii) above.  For instance, macroscopic events involving
\emph{infinite} sample spaces are, arguably, only gedankenexperiments.%
\footnote{The issue of physical infinity was already discussed in
note~\ref{f3b}.  Parker considers a `relativistic' version of the coin
flips that allegedly is `more realistic' \cite[p.\,14]{Pa21}.
However, even the relativistic approach relies on idealizing
assumptions such as an `eternal sequence' of coin flips that ``occur
in the same place at ten-second intervals'' or a `continuum' of such
sequences of coin flips.  As discussed in \cite[Section~2.3]{21b} an
assumption of a countable time-ordered model already involves a
full-fledged idealisation arguably lacking a physical instantiation.
Indeed, hypotheses such as time ordering or the use of the continuum
refer to a basic mathematical model belonging to the realm of what
Cartwright calls \emph{phenomenological model-building}, and not to
the pre-mathematical notion of a vast plurality of trials.}
Indeed, to the best of our current physical knowledge, it is unknown
whether there is such a thing as e.g., a physical infinite sequence.%
\footnote{On the other hand, microscopic events involving infinite
  sample spaces may arise in some mathematical models of quantum
  mechanics and statistical mechanics.}
%
%This sentiment is shared e.g.\ by Barrett.%
%
%\footnote{Barrett \cite[p.\;67]{barrett10}.}
%
%MK: Unfortunately, I did not find anything on page 67 in Barrett
%resembling the claim that ``there is no physical infinite sequence.

Thus, mathematical models featuring infinite sample spaces need to be
distinguished from physical phenomena involving vast pluralities of
trials, in order to avoid the risk of circularity and related
fallacies.

Some gedankenexperiments leading to the rejection of regular chances
are based upon implicit assumptions that may predetermine the
corresponding basic mathematical model.%
\footnote{For earlier discussions of such circularity see e.g.,
  Bottazzi and Katz \cite[Sections 2.3, 2.4, and~2.5]{21b} and
  Bottazzi et al.~\cite[section~3.5, pp.\;212--213]{19c}.}
We elaborate further on the distinction between physical phenomena and
their mathematical models in Section~\ref{cart2}.

\subsection{Basic mathematical models in probability}
\label{cart2}

The first step of building a basic model for the phenomenon of a vast
plurality of (independent) trials admits two significant
possibilities.  The possibilities that have been widely discussed in
the literature are
\begin{enumerate}
\item[]
(A)~$\N$ as a basic mathematical model for the vast plurality, and
\item[]
(B) a nonstandard integer number of trials as a basic mathematical
  model.
\end{enumerate}
Formalizing pre-mathematical vast plurality by means of a nonstandard
integer as in approach (B) would mean simply the following.  A
nonstandard positive integer is greater than any naive counting number
such as~$1,2,3$; in this sense, it formalizes the idea of a vast
plurality.%
\footnote{\label{f17}In the model-theoretic approach to analysis with
infinitesimals, the starting point is a hyperreal extension~$\astr$ of
the field~$\R$, featuring also a proper extension~$\astn$ of the
semiring~$\N$.  The nonstandard elements of~$\astn$ are unlimited
(hyper)integers.  However, there is an alternative axiomatic/syntactic
approach to analysis with infinitesimals, where infinitesimals are
found within the field~$\R$ itself, and nonstandard integers in~$\N$
itself.  The starting point of such an approach is an enrichment of
the language of set theory by the addition of a one-place predicate
`standard'.  A nonstandard integer in~$\N$ can plausibly serve as a
model for an infinite sample space because it is greater than every
naive integer such as~$1,2,3$.  For details see Section~\ref{appendix
  nelson}, as well as Bottazzi and Katz \cite[Section~3.2]{21b}.
Nonstandard integers are also exploited in the formalisation of
classical paradoxes, such as the sorites paradox and the Theseus ship
paradox; see Dinis \cite{Di23}.}
Only then does one apply a full-fledged mathematical theory, starting
with a suitable sample space and measure adapted to the parameters of
the basic mathematical model.%
\footnote{\label{f7}The mathematical theory includes e.g., the choice
  of the subsets of the sample space that represent possible outcomes.
  This collection is typically extended in a way that is
  mathematically convenient.  For instance, the adoption of
  real-valued probability measures usually leads to the choice of an
  algebra or a~$\sigma$-algebra that includes the desired events.  For
  a class of non-Archimedean models (namely, the hyperfinite models),
  every subset of the sample space can be regarded as possible and
  possible events are assigned a probability (see Section \ref{sec
    strenghening} and footnote~\ref{footnote internal}).}

Cartwright's insight concerning the importance of basic models is not
always taken into account by scholars in the foundations of
probability, when dealing e.g., with the features of a vast plurality
of trials.%
\footnote{The relation between pre-mathematical vast pluralities and
mathematical infinities was discussed in Section \ref{cart}.}

\subsection{The issue of countable additivity}

The counterpart of countable additivity in Archimedean models is
hyperfinite additivity in non-Archimedean models based on the
hyperreals (see Section~\ref{f3}).  Both are convenient mathematical
tools; in fact, hyperfinite additivity of a suitable internal counting
measure can be used to prove the countable additivity of, say, the
Lebesgue measure.  Neither countable additivity nor hyperfinite
additivity is based on alleged physical hypotheses; indeed, the
physicist R.T. Cox did not include countable additivity in the
postulates of probability theory \cite{Co46}.

In this context, it is instructive to analyze a pair of comments that
appear in a recent article by Isaacs et al.:
\begin{enumerate}
  \item
``Adopting a non-standard measure (in which infinitesimal
  probabilities may be assigned) is much the same as denying countable
  additivity, as it means that the standard components of
  probabilities need not be countably additive.''%
\footnote{Isaacs et al.~\cite[p.\;912]{Is22}.}
\item
``Strictly speaking, a non-standard measure would not respect
  countable additivity, as the notion of the countable addition of
  non-standard numbers is undefined.''%
\footnote{Isaacs et al.~\cite[note 24]{Is22}.}
\end{enumerate}
Here Isaacs et al.~fail to acknowledge the fact that nonstandard
measures possess properties parallel to countable additivity of the
Lebesgue measure when it comes to modeling vast pluralities of trials.
Furthermore, their claim (1) is imprecise: the standard part of a
hyperfinite measure may well be countably additive. Indeed, every
probabilistic model based on a $\sigma$-additive measure can be turned
into a hyperfinite regular model in a way that the corresponding
standard part has the same additivity as the original measure.
%
%\footnote{Such hyperfinite models can be obtained by Theorem~\ref{prop
%    models 2}; see also the discussion in Benci et
%  al.~\cite[Section~2]{bbd}.}
%
Thus, modeling with a hyperfinite measure does not require ``denying
countable additivity'' anymore than modeling with a real-valued
measure.

\section{The fair spinner}
\label{section spinner}

Barrett describes the fair spinner as follows:
\begin{enumerate}\item[]
[A] fair spinner [is] a perfectly sharp, nearly frictionless pointer
mounted on a circular disk. Spin the pointer and eventually it comes
to rest at some random point along the circumference, which we
identify with a real number in the half-open interval~$[0,1)$.%
\footnote{Barrett \cite[p.\;65]{barrett10}.}
\end{enumerate}

As we will see in Section \ref{section hypotheses}, Barrett, Parker,
and Pruss claim that the mathematical models for this spinner should
have a high degree of symmetry, namely ``invariance under rotations
\ldots\ and modular translations''%
\footnote{Parker \cite [pp.\;16--17]{Pa21}.}
%this was page 27 in the preprint version and referee X referred to
%it.
by an arbitrary real angle and real vector, respectively.  No known
models of this process are simultaneously regular and invariant under
these transformations.%
\footnote{More specifically, if a set~$A = \{T^n(x)\}$ (for~$n\in\N$)
  is infinite, then no measure defined for~$\{x\}$,~$A$, and~$TA$ can
  be regular and invariant; cf.~Parker \cite[p.\,17]{Pa21}.}
Accordingly, regular models cannot be invariant under rotations and
modular translations by an arbitrary real angle.  From these
considerations,
%Barrett, Parker, and Pruss claim to deduce that physical principles
%rule out the use of regular chances.  
they draw the conclusion that the hypothesis of invariance under
rotations and modular translations forces the choice of an Archimedean
model based on the Lebesgue measure.

In the next sections we discuss the role of invariance under these
transformations and discuss the extent to which both Archimedean
(A-track) and non-Archi\-medean (B-track) models of the spinner
satisfy the various hypotheses usually attributed to this
gedankenexperiment.  In doing so, we will address the question whether
the Lebesgue measure is the only Archimedean model satisfying the
alleged physical principles underlying the spinner.  Ultimately, we
argue that the current hypotheses for the spinner are not sufficient
to determine a unique model, whether among Archimedean models or among
non-Archi\-medean models.

Hyperreal probabilities have been repeatedly criticized for their
alleged lack of unique models.%
\footnote{See the discussion of such criticisms by Pruss in Bottazzi
  and Katz \cite{21b,21c} and references therein.}
For instance, Parker claims to detect `theoretical vices' in hyperreal
probabilities, namely ``dependence on extrinsic, haecceitistic
features, and, arguably, arbitrary and potentially misleading excess
structure.''%
\footnote{Parker \cite[p.\;9]{Pa21} (see further in
Section~\ref{s55}).}

Our analysis shows that a similar critique of non-uniqueness applies
to Archimedean models, as well.  As we show in Section \ref{sec
  A-track}, the opinion that a uniform chance over an interval has a
unique mathematical model is baseless.

We stress that the possibility of representing the same phenomenon by
means of distinct mathematical models is not necessarily a flaw, since
it enables the development of models with additional mathematical
features that may be relevant for various applications.%
\footnote{In many cases the advocates of Archimedean models routinely
rely on purely mathematical (rather than philosophical) features in
pursuit of their advocacy; such an attitude is reasonable and applies
equally to non-Archimedean models.}
%To the best of our knowledge, this mathematical
%  consideration has no philosophical counterpart in the debate on
%  models for uniform chances. We encourage further discussion on the
%  philosophical merits and shortcomings of such additional features in
%  Archimedean and non-Archimedean models.
%
Examples of such additional features for non-Archimedean models are
discussed in Section \ref{sec B-track}.%
\footnote{See also Bottazzi and Katz \cite[Section~2.5]{21c}.}

\subsection{Choosing hypotheses on the model}
\label{section hypotheses}

Let us recall Barrett's hypotheses on the uniform spinner
\cite[p.\;65]{barrett10}.
\begin{itemize}
	\item[(A)] The chance that the pointer initially at
          position~$a \in [0,1)$ will come to rest in a set~$S
            \subseteq [0,1)$ is independent of~$a$.
	\item[(B)] The possibility that the pointer initially at rest
          at position~$a$ will come to rest somewhere in a set~$S$ is
          independent of~$a$. 
	\item[(C)] The chances and possibilities are unaffected if the
          scale is rotated by an arbitrary real angle.
	\item[(D)] At least one point outcome is possible.%
\footnote{\label{footnote po1}Besides qualitative, comparative
  theories of chance that do not assign a specific probability to any
  outcome, we are not aware of any model, be it Archimedean or
  non-Archimedean, that accepts~(D) but does not assign a probability
  to some point outcome.  In the philosophical discussion of infinite
  processes, it appears that when one assumes hypothesis (D), one
  tries also to assign every point outcome a probability, if this is
  compatible with the other hypotheses on the physical problem under
  analysis.}
\end{itemize}
From these hypotheses Barrett deduces the following:
\begin{itemize}
\item[{\sy}$_{\R}$] The chance is invariant%
\footnote{Pruss refers to such invariance as ``symmetry''.  Hence the
  abbreviation {\sy}, used also in \cite[Section 2]{21c}.}
with respect to rotations by an arbitrary real angle (equivalently,
with respect to translations modulo~$1$ of any real vector);
\item[{\Un}] The chance is uniform;
\item[{\Po}] Every point outcome is possible
and is assigned a probability.%
\footnote{We have chosen to include the hypothesis that every point is
  assigned a probability in our formulation of {\Po}, since this
  is consistent with the established practice.  For further details,
  see footnote \ref{footnote po1}.}
\end{itemize}
As argued in \cite[Section 2]{21c}, hypothesis {\Un} assumes different
forms for Archimedean and non-Archi\-medean models. For Archimedean
models based upon real-valued measures on~$[0,1)$ it can be phrased as

\begin{itemize}
\item[{}\quad{\Un}$_{\,a}$] If two intervals have the same length,
  then they have the same chance.
\end{itemize}

For hyperfinite models, the appropriate definition is the following.
\begin{definition}
\label{f9}
  {\rm {\Un}}$_{\,h}$  If two internal sets have the same
  number of points, then they have the same chance.
\end{definition}

Barrett's description of the fair spinner constitutes a basic
mathematical model.%
\footnote{See note~\ref{f16} and Section~\ref{cart}.}
From this description, however, some authors skip directly to the
advanced mathematical tool of the Lebesgue measure without further
justification.  Indeed, this is assumed to be `the standard model'
of the spinner.%
\footnote{Barrett \cite[pp.\;67--68]{barrett10}.  Barrett only argues
  against the use of atomic measures on~$[0,1)$, but does not discuss
    the significance of the choice of the~$\sigma$-algebra of Lebesgue
    measurable sets.  Furthermore, he does not take into account the
    existence of other continuous models with the same desirable
    properties as the Lebesgue measure, as we discuss in
    Section~\ref{sec A-track}.}

Barrett does not elaborate on the relation between the notion of
possibility of an outcome and the property of a mathematical model of
assigning some probability to every possible outcome.  However, his
choice of the Lebesgue measure as `the standard model' of the spinner
enables him to assign a (null) probability to every point outcome.%
%
%\footnote{\label{footnote po}Furthermore, we are not aware of any
%  model, be it Archimedean or non-Archimedean, that accepts
%  {\Po} but does not assign a probability to some point
%  outcome. In the philosophical discussion of infinite processes, it
%  appears that {\Po} is often implicitly conflated with the
%  stronger statement ``every point outcome is assigned a
%  probability.''}
%

\subsection{Strengthening the hypothesis {\Po}}
\label{sec strenghening}

While Barrett focused on the hypothesis~{\Po} and on other events that
can be described as a disjoint union of countably many points, all of
which have a probability in the standard model, he does not mention
more general events.  However, we argue that endorsement of the
property that every point outcome is possible could arguably entail
endorsement also of the property that every set outcome is possible.
Namely, if we assume that each~$x\in[0,1)$ is possible, we could
  accept that also every~$S\subseteq[0,1)$ that includes~$x$ is
    possible.%
\footnote{This monotonicity argument resembles e.g., the monotonicity
  of logical consequence in classical logic. In that setting, if
 ~$\phi$ is a consequence of~$\Sigma$, then it is also a consequence
  of any set including~$\Sigma$ as a subset.}

For the purpose of our survey on Archimedean and non-Archimedean
models of the spinner, we find it relevant to discuss also the extent
to which they can satisfy the even stronger property
\begin{itemize}
\item[\textbf{Tot}] Every set outcome is assigned a probability.
\end{itemize}
This property is assumed as an axiom of Non-Archimedean Probability by
Benci et al.~\cite{benci2018}.%
\footnote{Taking into account that some possible outcomes might not be
assigned a probability, \textbf{Tot} is a stronger property than
``every set outcome is possible''.}
%On the distinction between possible events and events that are
%assigned a measure, see also footnote~\ref{footnote po}.
%
Note that neither {\Po} nor \textbf{Tot} requires the notion of
chance to be regular.%
\footnote{For example, the Lebesgue measure is total but not regular
  in the Solovay model; see note~\ref{f43}.}

In the sequel, we discuss how the current models are positioned with
respect to hypotheses {\sy}$_{\R}$, {\Un}, {\Po},
\textbf{Reg} and \textbf{Tot}.

\section{The A-track approach}
\label{sec A-track}

The A-track approach seeks to model the outcome of the uniform spinner
by means of an Archimedean measure (i.e., an~$\R$-valued measure).
For such measures, there is a well-known incompatibility result among
properties {\Un}, {\Po} and \textbf{Reg}.  We also
consider the property

\begin{itemize}
	\item[{\Un}$_p$] all points have the same measure,
\end{itemize}
implied by uniformity in the case of the sample space~$[0,1)$ as
  above.

\begin{proposition}\label{prop 1}
An Archimedean probability measure~$\mu$ on an infinite sample space
$X$ cannot satisfy simultaneously the hypotheses \emph{{\Un}$_p$},
% \footnote{In this general context, {\Un} should be formulated
%as ``points have the same measure''.}
\emph{{\Po}} and \emph{\textbf{Reg}}.
\end{proposition}
\begin{proof}
Hypotheses {\Un}$_p$, {\Po} and \textbf{Reg} together
entail that points have the same positive measure
$\varepsilon>0$. Then infinity of the sample space and finite
superadditivity of the measure are sufficient to conclude
that~$\mu(X)> n\varepsilon$ for each~$n$.  Since the measure is
Archimedean, these inequalities contradict~$\mu(X) = 1$.
\end{proof}

Proposition \ref{prop 1} entails that an Archimedean model for the
spinner satisfying {\Po} must reject at least one of the
properties {\Un} or \textbf{Reg}.  Barrett,%
\footnote{Barrett \cite[p.\;65]{barrett10}.}
Parker,%
\footnote{Parker \cite[pp.\;17--18]{Pa21}.}
and Pruss%
\footnote{Pruss \cite[Section~3.3]{Pr21a}.}
argue that physical principles suggest the use of a notion of chance
that also satisfies {\sy}$_{\R}$.  However, if one models the
outcomes of the spinners with the Archimedean
interval~$[0,1)\subseteq\R$, {\Po} and {\sy}$_{\R}$ are
  also incompatible with \textbf{Reg} (the proof is similar).

\begin{proposition}
\label{prop 2}
An Archimedean probability measure~$\mu$ on the interval~$[0,1)
  \subseteq \R$ cannot satisfy simultaneously the hypotheses
  \emph{{\Po}}, \emph{\textbf{Reg}} and
  \emph{{\sy}}$_{\R}$.
\end{proposition}

%This incompatibility result was proved by Parker.%
%
%\footnote{Parker \cite[Section\;8, p.\,17]{Pa21}.}

% but referee X seems to claim that Parker's proof does not imply this
% ``because the hypotheses do not imply that countable sets are
% assigned a probability''.

%Notice that this result requires assigning a probability also to
%countable unions of points and to intervals, a property that is
%stronger than {\Po} and weaker than \textbf{Tot}.

%anyway Proposition 4.2 assumes that the measure is Archimedean, so
%the proof is much simpler and one does not need Parker's argument.

Proposition \ref{prop 2} and the preference for {\sy}$_{\R}$
force the rejection of either {\Po} or \textbf{Reg}.%
\footnote{Dropping the hypothesis \textbf{Reg} might lead to measures
that satisfy {\Un}, {\Po} and possibly other properties, such as
{\sy}$_{\R}$, but do not assign a measure to finite sets (and also not
to countable sets if the measure is~$\sigma$-additive). To the best of
our knowledge, this kind of measure has not been considered as a model
of uniform chances.  For this reason, in the sequel whenever we assume
{\Po} we will usually also assume that points are assigned a
probability (but see Section~\ref{sec minimal model}).}
%Notice that this seems to be a common practice, as mentioned in
%footnote \ref{footnote po}.
%
In practice, one typically assumes {\Po}, forcing the rejection
of \textbf{Reg}.  Thus supporters of Archimedean models propose a
notion of chance based upon a particular measure that satisfies
hypotheses {\sy}$_{\R}$, {\Un}, {\Po}, namely the
Lebesgue measure.
%$L$.

\subsection{When is the Lebesgue measure total?}
\label{section is Lebesgue total}

If one rejects the full Axiom of Choice, then the Lebesgue measure can
also satisfy \textbf{Tot}.  Indeed, assuming that ``there exists an
inaccessible cardinal,'' Solovay provides a model for the real numbers
where the axiom of dependent choice (ADC) holds and every set of real
numbers is Lebesgue measurable.%
\footnote{\label{f43}Solovay~\cite{So70}.}
Thus, Solovay's model assumes the existence of a weakly inaccessible
cardinal.  Indeed, the consistency of the theory ZF+ADC+``all sets of
reals are Lebesgue measurable'' implies consistency of ZF+``there
exists an inaccessible cardinal.''  For an introductory discussion on
Solovay's model see Ciesielski \cite{howgood} and references therein.
See also Howard and Rubin \cite[pp.\,150--151]{Ho98}.%
\footnote{Mancosu and Massas assert that ``Solovay (1970) showed that
if there is an inaccessible cardinal, then there is a model of ZF in
which all sets of reals are Lebesgue measurable'' \cite{Ma24}.
However, this is a misstatement of Solovay's result.  Models of ZF
where all sets of reals are Lebesgue measurable were known
considerably earlier, such as the Feferman-Levy model from the 1960s.
Solovay's achievement was to build such a model for ZF+ADC.
Concerning the definition of the Lebesgue measure in ZF (without C),
see e.g., Kanovei and Katz \cite{17i}.}

Without assuming inaccessible cardinals, one can build a model of
ZF+ADC where the Lebesgue measure admits a translation-invariant,
$\sigma$-additive, total extension.  Such a model is due to Solovay
and was presented in detail by Sacks.%
\footnote{Sacks \cite[Theorem~4.26]{Sa69}.}

In order to prove~$\sigma$-additivity of the Lebesgue measure it is
sufficient to work with the axiom of countable choice (ACC).  Since
ADC entails ACC, in Solovay's model the Lebesgue measure is
$\sigma$-additive and satisfies \textbf{Tot.}

However, the theory ZF+ADC is unable to prove certain fundamental
theorems of set-theoretic mathematics,%
\footnote{For the distinction between ordinary mathematics and
  set-theoretic mathematics, see Simpson \cite{simpson}.}
such as Tychonoff's Theorem, the Prime Ideal Theorem or the
Hahn--Banach Theorem for general vector spaces.  Each of these
theorems requires a stronger version of choice than ADC.%
\footnote{For a discussion of some choice principles used in
  set-theoretic mathematics, we refer to Hrbacek and
  Katz~\cite[Section~1.1]{21e}.}
Meanwhile, by assuming even the Hahn--Banach Theorem, it is possible
to prove the existence of sets that are not Lebesgue measurable, so
that this measure would not be total.%
\footnote{For further details, we refer to Foreman and Wehrung
  \cite{Fo91} and Fremlin \cite{fremlin}.}
%

%and to \cite[Section~5]{19c}.

\subsection{Extensions of the Lebesgue measure in ZFC}

It turns out that, modulo a sufficiently strong choice principle that
entails the existence of nonmeasurable sets,%
\footnote{Usually, the results described in Propositions \ref{prop
    sigma additive extensions}, \ref{prop incompatibility} and
  \ref{a-track dream} are formulated in ZFC. It might be of interest
  to determine the minimal foundational framework necessary to prove
  these statements.}
the Lebesgue measure is not the unique measure satisfying the
hypotheses {\sy}$_{\R}$, {\Un} and {\Po}. Indeed,
the Lebesgue measure admits many extensions, as detailed in the next
proposition.

\begin{proposition}
\label{prop sigma additive extensions}
For every set~$A \subseteq [0,1)$ of Lebesgue outer measure~$1$ such
  that no countable set of translates of~$A$ covers any set of
  Lebesgue measure greater than~$0$, there is a~$\sigma$-additive
  extension of the Lebesgue measure on~$[0,1)$ to a
    translation-invariant measure for which~$A$ is a null set.
\end{proposition}
\begin{proof}
See e.g., Ciesielski and Pelc,%
\footnote{Ciesielski and Pelc \cite{extensions}.}
Ciesielski,%
\footnote{Ciesielski \cite[p.\;56]{howgood}.}
or Fremlin.%
\footnote{Fremlin \cite[Exercise 442Yc]{fremlin}.}
\end{proof}

However, there is no~$\sigma$-additive extension of the Lebesgue
measure that satisfies {\sy}$_{\R}$, {\Un} and
\textbf{Tot}.  When the sample space is the Euclidean space, we use
{\sy}$_{\R}$ to denote invariance under all isometries.

\begin{proposition}
\label{prop incompatibility}
The hypotheses \emph{{\sy}}$_{\R}$, \emph{{\Un}} and
\emph{\textbf{Tot}} are incompatible for an Archimedean
$\sigma$-additive~$\sigma$-finite measure on~$\R^n$.
\end{proposition}

This incompatibility result is due to Ciesielski and
Pelc~\cite{extensions}.%
\footnote{See also Ciesielski \cite[p.\;56]{howgood}, who notes that
  the case of~$\R$ had already been obtained by Harazi\v{s}vili
  \cite{Ha77}.  A short proof in the general case is due to
  Ciesielski~\cite{Ci90}.}

In particular, the extensions provided by Proposition \ref{prop sigma
  additive extensions} are not maximal.  Thus, assuming a sufficiently
strong choice principle, hypotheses {{\sy}}$_{\R}$, {{\Un}} and
{\textbf{Tot}} are incompatible for Archimedean models of physical
processes in~$\R^3$.  It turns out that the incompatibility result of
Proposition~\ref{prop incompatibility} is only due to the requirement
of~$\sigma$-additivity.

\begin{proposition}\label{a-track dream}
There exist finitely additive extensions of Lebesgue measure
on~$[0,1)$ that satisfy \emph{{\sy}}$_{\R}$, \emph{{\Un}} and
  \emph{\textbf{Tot}}.%
\footnote{This result is true for~$\R$ and~$\R^2$, but not for~$\R^n$
  with~$n\geq 3$. For a discussion of this limitation see
  \cite[p.\;57]{howgood}.}
\end{proposition}
\begin{proof}
See e.g., Royden.%
\footnote{Royden \cite[Chapter 10, Exercise 21]{royden}.}
This result was also proved via infinitesimal techniques by Bernstein
and Wattenberg~\cite{bernstein+wattenberg}.
\end{proof}

Note however that the finitely-additive measure provided by
Proposition~\ref{a-track dream} is not necessarily unique, since it
depends upon a nonconstructive principle (Royden's proof relies on the
Hahn--Banach theorem).%
\footnote{Proposition \ref{a-track dream} is incompatible with
H\'ajek's claim that, for an Archimedean model of the spinner, ``if we
want the probability to respect the symmetries {\ldots} then [the
  family of sets with an assigned outcome] cannot be the power set of
[the sample space].  Famously, certain subsets of [the sample space]
-- so-called non-measurable sets -- go missing''
\cite[Section~4]{staying}.  Namely, they don't ``go missing'' for the
measure of Proposition \ref{a-track dream}.}

In this section we saw that, assuming a sufficiently strong choice
principle, the hypotheses on the spinner do not single out a unique
notion of chance.

\begin{corollary}\label{a-track underdetermination}
	Hypotheses \emph{{\sy}}$_{\R}$, \emph{{\Un}} and
        \emph{{\Po}} or the stronger hypotheses
        \emph{{\sy}}$_{\R}$, \emph{{\Un}} and
        \emph{\textbf{Tot}} are insufficient to determine uniquely a
        notion of chance for the spinner.
\end{corollary}

\subsection{Assessment of the A-track models}
\label{sec conclusions atrack}

Advocates of Archimedean probability measures focus on the
representation of the spinner obtained via the Lebesgue measure.
Unless one works in models where all subsets of~$\R$ are Lebesgue
measurable, thus rejecting choice principles sufficient to entail
e.g., the Hahn--Banach Theorem, the Lebesgue measure does not satisfy
the hypothesis \textbf{Tot}.%, a consequence of {\Po}.
\footnote{It would be possible to adopt a notion of chance based upon
  the Lebesgue inner measure instead on its finitely additive maximal
  extensions. This notion of chance would be invariant under
  translations and total, but it would not even be finitely additive.}
With a sufficiently strong choice principle it is possible to find
non-unique Archimedean finitely-additive measures compatible with, but
different from the Lebesgue measure, satisfying {\sy}$_{\R}$,
{\Un} and {\Po} or the stronger {\sy}$_{\R}$,
{\Un} and~\textbf{Tot}.

Thus it appears that Archimedean models based on these hypotheses
suffer from the very non-uniqueness seen as a shortcoming of
non-Archimedean models by their critics.
%
%``theoretical vices'' that Parker claims to detect in hyperreal
%probabilities.  Indeed, in this section we have seen that the choice
%of a particular~$\sigma$-algebra of measurable sets that are assigned
%a probability or the specific value of the probability of a
%non-Lebesgue measurable subset of~$[0,1)$ constitute
%\begin{quote}
%  ``extrinsic, haecceitistic features, and, arguably, arbitrary and
%  potentially misleading excess structure.''%
%%
%\footnote{Parker \cite[p.\;9]{Pa21}.  Meanwhile, the alleged
%  ``excess structure'' attributed to hyperreal-based probabilities is
%  not present in an approach based on Nelson's Internal Set Theory, as
%  discussed in \cite[Section 2.8]{19c} and
%  \cite[Section~3.2]{21b} (see also Section~\ref{appendix
%    nelson}).}
%
%\end{quote}
% \footnote{Parker's haecceitistic critique of hyperreal models is
%addressed in Section \ref{s55}.}
Adopting an argument by Pruss against hyperreal probabilities, one
could assert that ``there is no reason to take one or another
extension as the home of the correct model of an agent’s credences.''%
\footnote{Pruss \cite[Section 2, p.\;780]{Pr21a}.}
Thus, it can be argued that the Archimedean models of the spinner are
underdetermined.%
\footnote{Similar underdetermination charges against non-Archi\-medean
  models of the spinner based upon hyperfinite measures were addressed
  in \cite{21b,21c}.}
One can usefully compare this form of underdetermination with the
non-Archi\-medean case.  In the latter, a class of internal models
(namely hyperfinite models) is uniquely determined by the sample
space.%
\footnote{See Section \ref{sec conclusions btrack} and
  \cite[Section~3.4]{21c}.}
By comparison, in the Archimedean case it appears difficult to single
out one particular feature that uniquely determines a model.

\subsection
{Possible path to unicity: minimal Archimedean model}
\label{sec minimal model}

It turns out that there does exist a particular feature that, together
with the hypotheses {\sy}$_{\R}$, {\Un} and {\Po}, is sufficient to
single out a unique Archimedean model.  This feature is minimality.
Namely, it is possible to define uniquely an Archimedean model as the
restriction of the Lebesgue measure to the algebra~$\mathcal{M}$
generated by all the intervals (including closed ones, to be able to
include points, as well) of~$[0,1)$.%
\footnote{If one considers only the algebra generated by half-open
intervals of the form~$[x,x+a)$, one obtains an even smaller model
  that, however, does not assign a probability to any finite set.}
%In light of the common practice of using models that assign every
%point outcome a probability, (see footnote \ref{footnote po}), we
%have focused on the minimal model with this property.}
This minimal model assigns to
%  every countable subset of~$[0,1)$ the measure~$0$ and to
finite unions of disjoint intervals the sum of their length.  It also
has the convenient mathematical feature of~$\sigma$-additivity.%
\footnote{Concerning this Archimedean model, we note the following.
Since~$\mathcal{M}$ is not a~$\sigma$-algebra, one needs to explain
what one means by~$\sigma$-additivity.  It takes the following form.
Let~$L$ be the Lebesgue measure.  If~$A_n \in \mathcal{M}$ for
every~$n \in \N$, if~$A_i \cap A_j = \emptyset$ for every~$i\ne j$ and
if~$\bigcup_{n\in\N} A_n \in \mathcal{M}$, then
~$L\left(\bigcup_{n\in\N}A_n\right) = \sum_{n\in\N} L(A_n)$. Notice
the additional hypothesis~$\bigcup_{n\in\N} A_n \in \mathcal{M}$ that
is not necessary when working with~$\sigma$-additive measures defined
on~$\sigma$-algebras.}

However, it turns out that the existence of such a minimal model
cannot be construed as an advantage of solely the Archimedean models.
Indeed, in Section~\ref{sec minimal na model} we show that it is
possible similarly to specify a non-Archimedean model that agrees with
the minimal Archimedean model and, in addition, satisfies
\textbf{Reg}.

%Perhaps a supporter of Archimedean models might want to go the other way around: since maximal models are not unique, they could choose the minimal one (namely, assigning measure zero to every interval and only assigning a measure to finite unions of intervals). This is once again overlooked by detractors of regular models and might be the only way to get a unique Archimedean model out of the suggested hypotheses. 

\section{The B-track approach}
\label{sec B-track}

The B-track approach seeks to model the outcome of the uniform spinner
by means of non-Archi\-medean measures.
%More precisely, not every non-Archi\-medean measure is expressive enough to represent such situations; see the discussion in \cite[Section\;3.4]{21b} and \cite[Section\;4]{21c}.
Currently, the most expressive non-Archi\-medean models are the
hyperfinite models based upon Robinson's theory of mathematics with
infinitesimals.%
\footnote{%
%For a brief introduction to Robinson's theory of mathematics
%  with infinitesimals, see Section~\ref{s6} and references therein.
Recall that the non-Archi\-medean probabilities obtained with
the~$\Omega$-limit approach of Benci et al.\ can also be obtained from
hyperfinite models \cite[Section\;3.2]{21c}.}
A detailed list of some non-Archi\-medean models for the spinner is
presented in \cite[Section 2.3 and Section 2.5]{21c}.  Let us recall
the definition of a hyperfinite model of the spinner.  For property
{\Un}$_h$ see Definition\;\ref{f9}.%
\footnote{Hyperreal fields used in this section require stronger
saturation properties than those of the basic construction outlined in
Section~\ref{f3}.}

\begin{definition}
\label{dhm}
A hyperfinite model of the spinner consists of a hyperfinite
set~$\Omega\subseteq\ns{[0,1)}$ such that~$\st\,\Omega =[0,1]$.%
\footnote{For the notion of hyperfinite set and of the standard part
  of a limited element of~$\astr$ see Section~\ref{f3}.}
The corresponding hyperfinite probability measure is~$P_{\Omega}(A) =
\frac{|A|}{|\Omega|}$ for every internal subset~$A \subseteq \Omega$,
so that~$P_{\Omega}$ satisfies the hyperfinite uniformity hypothesis
\emph{{\Un}}$_{\,h}$.  The hyperfinite set is chosen so as to
satisfy the following compatibility condition:
\begin{itemize}
\item[\emph{\textbf{Co}}]~$P_{\Omega}(\ns\!A\cap \Omega) \approx L(A)$
  for each Lebesgue measurable set~$A\subseteq[0,1)$.
	\end{itemize}
\end{definition}

% \footnote{This hypothesis is a weak form of {\Un}. The
% hypothesis {\Un} can be suitably adapted to the hyperfinite
% setting as discussed in \cite[Section 2.3]{21c}.}

Every hyperfinite measure can be turned into a hyperreal-valued
measure on the real interval~$[0,1)$, $\tilde{P} \colon \mathcal
  P([0,1)) \to \astr$, by
    setting~$\tilde{P}(A)=P_{\Omega}(\ns\!A\cap\Omega)$.  This is
    the basis for the approach by Benci et
    al.~\cite{benci2013,benci2018}, that further relies on a
    hyperfinite set~$\Omega$ satisfying the
    inclusion~$[0,1)\subseteq\Omega$
%
%hyperfinite set containing all standard reals
%
to ensure the validity of \textbf{Reg} for~$\tilde{P}$.%
\footnote{\label{footnote internal}On the other hand, the hyperfinite
  measures of the form~$P_{\Omega}$ always satisfy \textbf{Reg} on
  their domain (namely, the family of internal subsets of~$\Omega$).}

We review some properties of non-Archimedean models, in order to
highlight the differences with Archimedean models and to show
additional features that might be relevant for modeling more complex
phenomena. Moreover, we establish incompatibility results analogous to
those discussed in Section \ref{sec A-track} for the Archimedean
models.

\subsection{Lack of full invariance}

Hyperfinite models of the spinner are never invariant with respect to
rotations by all real angles.
%
%do not have as much symmetry as some Archimedean models.

\begin{proposition}
No hyperfinite probability for the spinner satisfies
\emph{{\sy}}$_{\R}$.
% for every internal subset of~$\Omega$.
\end{proposition}
\begin{proof}
This is a consequence of hyperfiniteness of the sample space~$\Omega$.
Indeed, there is~$x\in\Omega$ such that~$x+r\mod1\not\in\Omega$ for
some~$r\in\R$.  We provide a sketch of a proof which is most easily
stated in an axiomatic framework, whose finite nonstandard sets
correspond to hyperfinite sets of the model-theoretic frameworks (see
Section~\ref{appendix nelson}).  Let~$S^1$ be the unit circle of
length~$2\pi$.  Consider a nonempty finite subset~$F\subseteq S^1$.
Assume~$F$ is invariant under all rotations by standard angle.
Let~$G$ be the group of all rotations preserving~$F$.  By
assumption,~$G$ contains all standard rotations.  Since~$F$ is
finite,~$G$ must be a finite cyclic group, necessarily of nonstandard
order~$H$.  But then all elements of~$G$ are of the
form~$e^{\frac{2\pi i k}{H}}_{\phantom{I}}$, i.e., rotations by~$2\pi$
times a rational angle.  Since not all standard real numbers are of
that form, we obtain a contradiction.
\end{proof}

Indeed, working with hyperfinite models%
\footnote{Such as those discussed by Bottazzi and Katz
  \cite[Section~2.3]{21c}.}
suggests that it may be useful to relax the symmetry condition to the
following ``discrete'' condition:
\begin{itemize}
\item[{\sy}$^{\phantom{I}}_{|\Omega|}$] The chance is invariant with
  respect to rotations by an angle of the
  form~$e^{\frac{2{\pi}ik}{|\Omega|}}_{\phantom{I}}$, i.e., rotations
  by~$\frac{2\pi k}{|\Omega|}$ (equivalently, with respect to
  translations modulo~$1$ of vectors of the form
  $\frac{k}{|\Omega|}$), where~$|\Omega|$ is the (internal) number of
  elements in~$\Omega$, and~$k=1,2,3,\ldots,|\Omega|$.
\end{itemize}
If~$\Omega$ is obtained by picking equally spaced points in the
hyperreal interval~$\ns{[0,1)}$ and if~$|\Omega| = m!$ for some
  unlimited~$m \in \astn\setminus \N$, this condition is sufficient
  to entail the following property:
\begin{itemize}
\item[{\sy}$^{\phantom{I}}_{\Q}$] The chance is invariant with
  respect to rotations by an arbitrary rational angle (equivalently,
  with respect to translations modulo~$1$ of any rational
  vector).%
\footnote{This is due to the fact that, for every~$m \in
    \astn\setminus\N$, every rational number between~$0$ and~$1$ can
    be represented as~$\frac{k}{m!}$ for some~$k \in \astn$,~$k \leq
    m!$.}
\end{itemize}
However, neither {\sy}$^{\phantom{I}}_{|\Omega|}$ nor
{\sy}$^{\phantom{I}}_{\Q}$ entail {\sy}$_{\R}$ on a sample
space that is not a subset of~$\R$.  Moreover, some hyperfinite models
do not satisfy~{\sy}$^{\phantom{I}}_{|\Omega|}$.%
\footnote{Recall that, if~$\Omega \supset [0,1)$, then there exists~$x
  \in \Omega$ and~$k \leq |\Omega|$ such that~$x+\frac{k}{|\Omega|}
  \not \in \Omega$.}

\subsection{Compatibility results}

Despite the incompatibility of {\sy}$_{\R}$ with non-Archi\-medean
models, we have the following compatibility results.  For the property
{\Un}$_{\,h}$ see Definition \ref{f9}.

\begin{proposition}
%
%\label{prop models 2}
%
There is a hyperfinite set~$\Omega \subseteq \ns{[0,1)}$ such
  that~$P_{\Omega}$ satisfies \emph{\textbf{Reg}},
  \emph{\textbf{Tot}}, \emph{{\Un}}$_{\,h}$, \emph{\textbf{Co}} and
  \emph{\sy}$^{\phantom{I}}_{\Q}$.
\end{proposition}

\begin{proof}
See e.g., Bottazzi and Katz%
\footnote{Bottazzi and Katz \cite[Section 2.3]{21c}.}
and references therein.
\end{proof}

If one renounces the invariance with respect to rotations by an
arbitrary rational angle, it is possible to improve on other
conditions that might be relevant in more sophisticated mathematical
models.

An example is provided by compatibility not only with the Lebesgue
measure, but also with the Hausdorff~$t$-measures for every
$t\in[0,+\infty)$.  This feature shows that hyperfinite measures are
able simultaneously to represent uncountably many Archimedean
measures.  In addition, it might be relevant e.g., in models of
Brownian motion or for analysis on fractals.  For an example of
hyperfinite techniques in this field, see e.g.,
\cite{lindstrom,lindstrom 2}.  Thus, in \cite{lindstrom}, the
Hausdorff dimension of the fractal is related to the asymptotics of
the number of eigenvalues of the infinitesimal generator of the
Brownian motion.  The most recent work in probability and statistics
using nonstandard analysis (NSA) methods includes Alam and Sengupta
\cite{Al23} and Duanmu et al.~\cite{Du23}.

\begin{proposition}
\label{prop models 2b}
There is a hyperfinite set~$\Omega \subseteq \ns{[0,1)}$ such
  that~$P_{\Omega}$ satisfies \emph{\textbf{Reg}},
  \emph{\textbf{Tot}}, \emph{{\Un}}$_{\,h}$, \emph{\textbf{Co}} and
\begin{itemize}
\item[\emph{\textbf{Co}}$^{\phantom{I}}_{H}$]~$P_{\Omega}$ is coherent
  with the Hausdorff~$t$-measures over~$[0,1)$ for
    every~$t\in[0,+\infty)$ in the sense that for
      every~$H^t$-measurable~$A, B \subseteq [0,1)$ (where
        $H^t(B)\in(0,\infty)$), the sets~$A_{\Omega} =\asta \cap
        \Omega$ and~$B_{\Omega} = \astb \cap \Omega$ satisfy the
        relations
		$$\frac{H^t(A \cap B)}{H^t(B)} \approx \frac{P_\Omega(A_\Omega\cap
			B_\Omega)}{P_\Omega(B_\Omega)} = P_\Omega(A_\Omega|B_\Omega).$$ 
	\end{itemize}
\end{proposition}
\begin{proof}
This compatibility result is a consequence of the main theorem of
Wattenberg \cite{watt}.%
\footnote{See also Bottazzi and Katz \cite[Section 2.5]{21c}.}
\end{proof}

Another compatibility result improves on the invariance of the
hyperfinite measure in an algebra of subsets of~$[0,1)$ that does not
  include any nonempty null sets.

\begin{proposition}
\label{prop models 1}
For every algebra of Lebesgue measurable sets~$\mathfrak B$ such that
\begin{enumerate}
\item 
\label{551}
for every~$A \in \mathfrak{B}$ either~$A = \emptyset$ or the Lebesgue
measure of~$A$ is positive;
\item~$\mathfrak B$ includes all intervals of the form~$[a,b)
  \subseteq [0,1)$;
\end{enumerate}
there is a hyperfinite set~$\Omega \subseteq \ns{[0,1)}$ such
  that~$P_{\Omega}$ satisfies \emph{\textbf{Reg}},
  \emph{\textbf{Tot}}, \emph{{\Un}}$_{\,h}$, \emph{\textbf{Co}} and
  the property
\begin{itemize}
\item[\emph{{\Un}}$^{\phantom{I}}_{\mathfrak B}$]~$P_\Omega$ is
  uniform over sets of~$\mathfrak B$,
  i.e.,~${P}(\ns\!A\cap\Omega)={P}(\ns B \cap \Omega)$ whenever~$A, B
  \in \mathfrak B$ satisfy~$L(A) = L(B)$.
\end{itemize}
\end{proposition}
\begin{proof}
This compatibility result is a consequence of Theorem 2.2 of Benci et
al.~\cite{bbd} applied to the Lebesgue measure.%
\footnote{See Benci et al.\ \cite[Section 3]{bbd2}.  See also
  \cite[Section~2.5]{21c}.}
\end{proof}

Note that the measure $P_{\Omega}$ provided by Proposition~\ref{prop
  models 1} is defined on all internal subsets of $\Omega$.

Since most symmetry arguments for non-regular probabilities involve
non-empty null sets, hypothesis \eqref{551} of Proposition \ref{prop
  models 1} can be regarded as a weakness of the compatibility
condition {\Un}$^{\phantom{I}}_{\mathfrak B}$.  Still, the condition
is often overlooked by the detractors of infinitesimal probabilities
in the discussion of the properties of non-Archimedean models of
physical phenomena.%
\footnote{In their assessment of hyperreal chances, neither Parker nor
  Pruss addresses the additional invariance provided by
  {\Un}$^{\phantom{I}}_{\mathfrak B}$.}
Notice that this condition entails invariance with respect to
rotations by an arbitrary real angle for sets in~$\mathfrak B$
(namely, a restricted form of {\sy}$_{\R}$), improving on the
weaker translation invariance of other hyperfinite measures.

Additionally, due to hypothesis (2) on~$\mathfrak{B}$, a
non-Archi\-medean model
satisfying~{\Un}$^{\phantom{I}}_{\mathfrak B}$ satisfies also
the Archimedean notion of uniformity ``intervals of the same length
have the same chance.''%
\footnote{See Definition~\ref{f9}.}

The results discussed above are sufficient to obtain the following
well-known result, analogous to Corollary \ref{a-track
  underdetermination}.

\begin{corollary}\label{b-track underdetermination}
	Hypotheses \emph{\textbf{Reg}}, \emph{\textbf{Tot}},
        \emph{{\Un}}$_{\,h}$ and \emph{\textbf{Co}} are
        insufficient to uniquely determine a probability model of
        chance for the spinner.
\end{corollary}

However, as already pointed out in \cite[Section 3.4]{21c},
uniform hyperfinite models of the spinner are uniquely determined by
the hyperfinite sample space~$\Omega$ (see also Definition \ref{dhm}).

\subsection{Infinitesimal models of spinner without Axiom of Choice}
%in ZF+ADC
\label{section SCOT}

In Section~\ref{section is Lebesgue total} we mentioned Solovay's
model of~$\R$ where the Lebesgue measure satisfies~{\sy}$_{\R}$,
{\Un} and \textbf{Tot}, and furthermore ADC is satisfied.

Hrbacek and Katz developed a foundational framework SCOT, a subsystem
of both Nelson's IST and Hrbacek's HST (for a brief discussion of IST
and HST see Section~\ref{appendix nelson}).  The theory SCOT is a
conservative extension of ZF+ADC.%
\footnote{Conservativity means that SCOT proves the same theorems as
ZF+ADC.}
In SCOT, it is possible to develop a (hyper)finite counting measure
satisfying \textbf{Reg}
%\textbf{Tot}, {\Un}$_{\,h}$, {\sy}$^{\phantom{I}}_{\Q}$
and \textbf{Co}.%
\footnote{The construction is detailed in \cite[Section~3]{21e}.}
%
%Thus the development of hyperfinite models of the spinner does not
%require stronger choice principles than ZF+ADC.
%
%This situation is summarized in Table \ref{t6111}, showing also the
%weaker theory SPOT and the stronger BST.

Furthermore, the axiom \textbf{Tot} can be satisfied, as well, in the
following sense.  Recall that the Lebesgue measure in Solovay's model
satisfies \textbf{Tot}.  For the purposes of the following theorem, we
will use \textbf{Tot} to denote the assertion that the Lebesgue
measure is total.

\begin{theorem}
The theory\, {\rm SCOT+\textbf{Tot}} is a conservative extension of
the theory {\rm ZF+ADC+\textbf{Tot}}.
\end{theorem}

\begin{proof}
This follows from the result ``SCOT is a conservative extension of
ZF+DC'' just as in \cite[Section 8.6]{21e}, where it was
shown that SCOT+CH is a conservative extension of ZF+CH.
\end{proof}

Thus one can do infinitesimal analysis in the assumption of
\textbf{Tot} just as one can do Archimedean analysis in that
assumption, in Solovay's model.

\subsection{Assessment of the B-track models}
\label{sec conclusions btrack}

One of the common contentions concerning hyperfinite models is that
they are not uniquely determined.%
\footnote{See e.g., Barrett \cite[pp.\;71--72]{barrett10}, Pruss
  \cite[Section 3]{Pr21a}).}
In response, Bottazzi and Katz showed in \cite[Section~2.3]{21c} that
distinct hyperfinite models are compatible in a sense that resembles
the isomorphism of Archimedean models based upon the representation of
the spinner with a different sample space, such as~$[0,2\pi)$
  or~$[0,360)$.  Moreover, hyperfinite representations of the spinner
    are uniquely determined by the hyperfinite set~$\Omega$.  In
    contrast, the Archimedean models discussed in Section \ref{sec
      A-track}, dependent on nonconstructive principles (such as the
    Hahn--Banach theorem), cannot be characterized as easily.

Barrett's critique of the `sprinkle spinner', namely a hyperfinite
spinner that contains all the real numbers, is similarly flawed.  The
`sprinkle spinner' is based on the measure~$\tilde{P}$ obtained from
a hyperfinite model where the sample space
satisfies~$[0,1)\subseteq\Omega$.  Barrett claims that this model
  features ``as much translation invariance as can be gotten''%
\footnote{Barrett \cite[p.\;72]{barrett10}.}
(namely, {\sy}$^{\phantom{I}}_{\Q}$).  However, condition
{\Un}$^{\phantom{I}}_{\mathfrak B}$ discussed in Proposition
\ref{prop models 1} (developed some years after Barrett's assessment)
shows that indeed it is possible to obtain an additional translation
invariance on a relevant family of subsets.

%A detractor of non-Archi\-medean probabilities could object that
%alleged physical principles should entail full translation invariance
%for all sets, and not only for a specific family. However, in Section
%\ref{sec A-track} we have already argued that full translation
%invariance is still not sufficient to single out a unique Archimedean
%model.
%Moreover recall that, in the context of probability measures,
%Proposition \ref{prop models 1} shows that the only obstacle to full
%translation invariance of a hyperfinite measure compatible with the
%Lebesgue measure%
%
%\footnote{Or indeed any real-valued measure.}
%
%is the presence of null sets.%
%
%\footnote{See also the discussion in Benci et al.~\cite[p.\;5]{bbd}.}
%

Finally, Barrett observes that the sprinkle spinner is an external
model, since it relies on a probability of the form~$\tilde{P}$
instead of~$P_{\Omega}$.  Accordingly, he claims that it ``represents
a big loss of instrumental virtue'' over internal models.%
\footnote{Barrett \cite[p.\;73]{barrett10}.}
Indeed, external entities are not subject to the transfer principle.%
\footnote{See Section~\ref{appendix a2}.}
However, there are some external models, such as the sprinkle spinner,
Benci's non-Archi\-medean probabilities and the Loeb measures, that
can be obtained from suitable internal models.  In the case of the
sprinkle spinner and Benci's non-Archi\-medean probabilities, the only
external object that is used is the domain of~$\tilde{P}$.%
\footnote{See also the discussion in \cite[Sections 3.2
    and~3.4]{21c}.}
Thus the `instrumental virtue' provided by the original internal
probability~$P$ is still present, contrary to Barrett's claim.

\subsection{A regular counterpart of the minimal Archimedean model}
\label{sec minimal na model}

As we saw in Section \ref{sec minimal model}, the hypotheses
{\sy}$_{\R}$, {\Un} and {\Po} are sufficient to single out a unique
\emph{minimal} Archimedean model.  It turns out that an analogous
situation exists for non-Archimedean models.  Start by applying
Proposition~\ref{prop models 1} with the choice of~$\mathfrak{B}$ as
the algebra generated by half-open intervals of the form~$[x,x+a)$.
  The corresponding~$P_\Omega$ is defined on all internal subsets of
  $\Omega$.  We restrict $P_\Omega$ to sets of the form
  $\asta\cap\Omega$, with $A$ belonging to the algebra~$\mathcal{M}$
  generated by intervals (not necessarily half-open).  Then the
  resulting measure is regular and it is compatible (in the sense of
  property {\co} of Definition \ref{dhm}) with the minimal Archimedean
  measure over the algebra generated by half-open intervals.  Thus the
  minimal Archimedean model admits a coherent non-Archimedean regular
  counterpart that has the property {\Un}$^{\phantom{I}}_{\mathfrak
    B}$, equivalent to the restriction of {\sy}$_{\R}$ to the algebra
  generated by half-open intervals. In other words, the regular
  non-Archimedean model has the same amount of translation invariance
  as the minimal Archimedean model modulo taking into account the
  existence of null sets in~$\mathcal{M}$.

As a consequence, recovering uniqueness via minimality is possible in
both Archimedean and non-Archimedean settings.  The non-Archimedean
model, while dependent on the choice of~$\Omega$, also enables
regularisation that does not feature appreciable `untenable
asymmetry'.%
\footnote{Parker \cite[p.\;2]{Pa21}.}
%
% or other alleged extrinsic features (suppressed as per item 53 in
%response.tex)
By
%Proposition \ref{prop models 2},
\cite{Bo21}, this result can be obtained within the same foundational
framework of the minimal Archimedean model~$\mathcal{M}$ of the
spinner.

\section{Parker's haecceitistic critique of regular chances}
\label{s55}

Commenting on the exchange between Williamson \cite{Wi07} and
Weintraub \cite{weintraub}, Parker claims the following:
\begin{enumerate}\item[]
``[T]he latter sort of system [where an infinite sequence of heads has
    infinitesimal chance] has clear \emph{theoretical vices}: The
  additional complexity of hyperreal numbers and of dependence on
  extrinsic, haeccetitistic [sic] features, and, arguably, arbitrary
  and potentially misleading excess structure (Parker, 2013, 2019;
  Pruss, 2013, 2021a).''%
\footnote{Parker \cite[p.\;9]{Pa21}; emphasis added.}
\end{enumerate}
Parker's claim of `theoretical vices' has at least four distinct
components:
\begin{enumerate}
[label={(Pa\theenumi)}]
\item
\label{i1}
an undesirable complexity of the hyperreal numbers;
\item
\label{i2}
dependence on haecceitistic features;
\item
\label{i3}
arbitrary and misleading excess structure;
\item
\label{i4}
citation of four articles in the literature: Parker \cite{Pa13},
\cite{Pa19} and Pruss \cite{Pr13}, \cite{Pr21a}.
\end{enumerate}
We will analyze items \ref{i1} and \ref{i3} in Section~\ref{s61}, item
\ref{i2} in Section~\ref{f87}, and item~\ref{i4} in Section~\ref{s63}.

\subsection{`Undesirable complexity' and `excess structure'}
\label{s61}

Let us first consider items \ref{i1} and \ref{i3} (we will return to
item \ref{i2} below).  Parker's critique of NSA approaches to
probability has a laser focus on the extension approaches, when~$\R$
is extended to~$\astr$.  However, in addition to such a
model-theoretic approach, there are also axiomatic/syntactic
approaches, such as those of Hrbacek, Nelson, and others.%
\footnote{See Section \ref{f87} and Section~\ref{appendix nelson}.}
Axiomatic approaches are typically based on a \mbox{st-$\in$}-language
rather than the~$\in$-language of traditional set theories (ZF, ZFC,
etc.).  The fundamental point is that these approaches are
conservative over the traditional frameworks.  In such approaches,
infinitesimals are found within~$\R$ itself rather than in an
extension thereof (they have not been detected by traditional
mathematicians who do not use the predicate `st').  Since one is still
working in the field~$\R$, there is little room for criticisms of
`excess structure' type.

In the axiomatic approach to NSA following Nelson, the resource
provided by the predicate st was there all along in~$\R$; it was
merely not used by classically trained mathematicians.  From this
perspective, a critic of the axiomatic approach to NSA as applied to
physical modeling would not be attacking the addition of `excess
structure', but rather the use of a particular type of resource
already present in~$\R$.  While few scholars would object to
exploiting e.g., the well-known resource of~$\R$ known as
completeness, Parker could still press his objection to the use of
another type of resource of the real numbers, namely the predicate st.
However, such an objection is less compelling than criticizing an
actual extension of the traditional number system, with the attendant
risks of non-uniqueness, underdetermination, etc., since few scholars
would attack~$\R$ as non-unique or underdetermined.

Thus Parker's criticism of `theoretical vices', `undesirable
complexity', and `arbitrary and misleading excess structure' appears
to be dependent on a specific (namely model-theoretic) approach to
NSA.  The fact that the criticism is off-target for the axiomatic
approaches indicates that Parker's criticism does not address the
substance of the infinitesimal approach; see further in
Section~\ref{f87}.

\subsection{Axiomatic approaches}
\label{f87}

To illustrate the fact that Parker's criticism misses its target,
consider the problem of an infinity of heads for an infinity of coin
tosses (referred to in Parker's comment quoted above).  It turns out
that this problem has a solution in the framework of the axiomatic
theory BSPT$'$,%
\footnote{See Hrbacek and Katz \cite[p.\,10]{21e}.}
where the standard reals can be included in a finite (nonstandard)
set~$F$,%
\footnote{The technical term for the nonstandard integer number of
elements of~$F$ is \emph{unlimited}.  An unlimited number is greater
than every naive integer~$1,2,3,\ldots$\ and in this sense can be said
to model an `infinity of coin tosses'.}
and assigning probability
\begin{equation}
\label{e51}
\frac1{|F|}>0
\end{equation}
to each toss, ensuring regularity over standard reals.

The theory BSPT$'$ is conservative over ZF, as shown by Hrbacek and
Katz;%
\footnote{Hrbacek and Katz \cite[p.\,10]{21e}.}
see also \cite{23d}, \cite{23e}.  As is well known, the theory ZF does
not prove Banach--Tarski theorem (paradox).  Meanwhile, Pruss claims
that
\begin{quote}
  ``Regular probability comparisons imply the Banach–Tarski Paradox''
\end{quote}  
(the title of his article!).%
\footnote{Pruss \cite[p.\;3525]{Pr14}.}
Seeing that BSPT$'$ enables regularity, Pruss's claim must depend on
unstated additional hypotheses.

Since in axiomatic approaches to NSA one is working in the field~$\R$,
the remarks above may constitute a rebuttal of item \ref{i2} as well,
provided one specifies what Parker might mean by his criticism
concerning `haecceitistic features'.  The traditional meaning of
haecceitism in philosophy is of little help here.  One of the
definitions in the Standford Encyclopedia of Philosophy, in an article
authored by Sam Cowling,%
\footnote{See \url{https://plato.stanford.edu/entries/haecceitism}}
runs as follows:
\begin{quote}
Modalist Anti-Haecceitism: Necessarily, the world could not be
non-qualitatively different without being qualitatively different.
\end{quote}
Since such a definition does not advance us much in understanding
Parker's objection, we will adopt the following working definition
that seems to match Parker's intention:
\begin{enumerate}\item[]
haecceitistic properties are those that depend on the particular
elements of a set rather than their qualitative properties
(analogously, a haecceitistic property would describe a situation
where, say, the mass of a collection of electrons would depend not
only on how many electrons are present or their momenta, but on which
specific electrons they happen to be).
\end{enumerate}
Relative to such an interpretation of Parker's remarks, one can point
out that, as far as standard real numbers are concerned, they are all
assigned the same probability~\eqref{e51} in the approach outlined
above, and therefore the probability of a collection will be
independent of which specific standard reals happen to be present.  As
far as nonstandard numbers are concerned, they are viewed as ideal
elements in the model, with no claimed physical counterpart.%
\footnote{The existence of such ideal entities in virtually any
mathematical theory of interest was already pointed out by David
Hilbert over a century ago.}

A persistent theme in Parker is the identification of `haecceitistic'
features in mathematical models.  Parker contrasts them with
`\emph{intrinsic}, \emph{qualitative}' ones.%
\footnote{Parker \cite[p.\;3]{Pa21}; emphasis in the original.}
He further argues that
\begin{enumerate}\item[]
[A] system of chances or credences that depend on extrinsic or
haecceitistic factors would be difficult to apply because the
discernible, local properties of a system would not be sufficient to
determine them.%
\footnote{Parker \cite[pp.\;7--8]{Pa21}.}
\end{enumerate}
Among such haecceitistic features, Parker mentions set inclusion,
`parthood' and timing between events in gedankenexperiments, in
addition to the critique of hyperreal probabilities mentioned in
Section~\ref{sec conclusions atrack}.
Parker goes on to attribute the distinction between intrinsic and
extrinsic features to Williamson:
\begin{enumerate}\item[]
On a plausible view, which William\-son seems to take for granted, the
chances and credences of events should be determined by their
\emph{intrinsic}, \emph{qualitative} properties and circumstances, not
extrinsic or haecceitistic features.  The differences Weintraub finds
in William\-son's events appear to be extrinsic and haecceitistic, and
hence to have no force against Williamson, given that view.%
\footnote{Parker \cite[p.\;3]{Pa21}; emphasis in the original.}
\end{enumerate}
However, neither Williamson \cite{Wi07} nor Weintraub
\cite{weintraub} are concerned with identifying extrinsic or
haecceitistic features in the mathematical models of the phenomenon
they discuss (namely, the gedankenexperiment of an infinite sequence
of fair coin tosses, discussed also in
\cite{barrett10,bbd2,benci2018,21b,Pr21a}).
Among the alleged extrinsic features (or `theoretical vices'),%
\footnote{Parker \cite[p.\;3]{Pa21}.}
Parker identifies models that are not `perfectly symmetric' under
isometries based on the real numbers.  The assumption here is that
isometries based on the real numbers (and not based e.g., on the
rational numbers or on hyperreals) necessarily constitute intrinsic
features of some physical phenomena.  Such an assumption is not based
on convincing evidence.  For a related discussion, see \cite{19c}.

As we saw in Section \ref{sec conclusions atrack}, the intuitive idea
that physical principles including symmetry under isometries of the
real line are sufficient to uniquely determine a model of chance
%
%that does not depend on extrinsic features
%
is refuted by a careful analysis of the Archimedean models.
We welcome Parker's realisation that there is a problem with a blithe
assumption of a time-ordered sequence of trials.  Such an assumption
is a thinly veiled attempt to portray as `physically motivated' the
endorsement of the set~$\N$ as the obligatory formalisation of the
intuitive vast plurality of trials, entailing also an endorsement of a
\emph{cardinality} rather than an \emph{unlimited number} as the index
set.  Indeed, Parker rejects such a \emph{time-ordered} idea.
%
%(Note: I am not sure how Parker would reconcile this progress with his
%apparent endorsement of physical infinite sequences in his abstract.
%Perhaps this was a tactical maneuver to get the paper accepted.  Or
%perhaps it was written at a different stage of the development of his
%paper, and constitutes a fall-back to the old prejudices.  What do you
%think?  Possibly he could argue that he is merely using "infinite
%sequence" in the generic sense of our "vast plurality".  Wiliamson
%likely did not mean that, but perhaps Parker did.  We may therefore
%want to tone down our criticism of that.)
%
Meanwhile, a natural alternative to the \emph{time-ordered events}
model is the \emph{simultaneous nonstandard number of events} model.
Here removing a single trial changes the probability calculated over
the remaining trials, resolving the claimed paradoxes.
The term \emph{haecceitism} from metaphysics, as used by Parker, is
designed to include a criticism of both the time-ordered idea and the
`subset' idea.  Since the technical concept of \emph{subset} is
strictly speaking not needed, and it suffices to work with a more
generic idea of parthood, Parker includes `parthood', as well.%
\footnote{In a more recent paper, Parker rejects this position and
  advocates \emph{in favor of} parthood \cite[Section\;5]{No22}.  More
  specifically, he advocates the use of chances that satisfy the
  Euclidean maxim `the whole is greater than the part'.}
Thus Parker makes short shrift of the original metaphysical meaning of
haecceitism, and proceeds to assign to it a series of additional
meanings that happen to be tailor-made for a critique of regular
models.  Such an ad hoc use of the metaphysical term involves an
equivocation on its meaning, rendering Parker's arguments
inconclusive.

\subsection{Literature and corrections}
\label{s63}

We now turn to Parker's final item \ref{i4} involving the citation of
four articles in the literature.  It is not entirely clear what
Parker's intention was here.  Since he probably did not mean that his
arguments in \cite{Pa21} had already been presented in those articles,
we will assume that those articles argue complementary points that
complete an alleged rebuttal, presented in~\cite{Pa21}, of
non-Archimedean probabilities.  In such case we shall point out that
these earlier articles already received a rebuttal in four recent
articles.%
\footnote{See (\cite{benci2018} 2018), (\cite{19c} 2019), (\cite{21b}
2021), (\cite{21c} 2021).}

The publication of Parker's text in a philosophy journal fits into a
recent tendency for papers critical of non-Archimedean probabilities
both to contain increasingly technical mathematical content and to
appear in philosophy journals whose refereeing corps is of
questionable competence to evaluate such content.  Such an endemic
problem results in serious mathematical errors being published.  

Thus, in \emph{European Journal for Philosophy of Science}, Norton's
text \cite{No18a} had to be followed by a correction in \cite{No18b}.

In \emph{Synthese}, Pruss's text \cite{Pr21b} was followed by a
correction in \cite{Pr22}.  

In a separate article in the same journal and year, Pruss claims that
\begin{enumerate}\item[]
[B]ecause any regular probability measure that has infinitesimal
values can be replaced by one that has all the same intuitive features
but other infinitesimal values, the heart of the arbitrariness
objection remains.%
\footnote{Pruss \cite[Abstract]{Pr21a}.}
\end{enumerate}
However, the other measures, obtained by rescaling the infinitesimal
part of the measure, are not internal and therefore irrelevant to
hyperreal modeling.%
\footnote{For further details, see Bottazzi and Katz \cite{21b}.}

An additional dubious claim by Pruss in \emph{Synthese} \cite{Pr14}
appeared in 2014, as already discussed in Section \ref{f87}.

In \emph{Studies in History and Philosophy of Science}, Parker defines
the concept of \emph{Label independence} as follows: ``All true
statements pertinent to the chances of different outcomes remain true
when the labels are arbitrarily permuted,'' and \emph{Containment}:
``If~$A$ is a proper subset of~$B$, then the chance of~$A$ is strictly
less than that of~$B$, i.e.,~$A\subset B\implies \text{Ch}(A) <
\text{Ch}(B)$,'' and goes on to claim the following:
\begin{enumerate}\item[]
[A]las, a chance
function that satisfies label independence over an infinite sample
space cannot satisfy containment.%
\footnote{Parker \cite[p. 28]{Pa20}.}
\end{enumerate}
However, Parker's claim depends on his assumption that infinity is
modeled by~$\N$, i.e., a countable cardinality; namely a specific
choice of a \emph{basic mathematical model} (see Section~\ref{cart}).
Meanwhile, an alternative is to model infinity by a nonstandard
integer together with a choice of a finite set%
\footnote{Here we work in an axiomatic approach to NSA; see
Section~\ref{f87}.}
containing all standard integers as above, both \emph{label
independence} and \emph{containment} hold, contrary to Parker's
claim.%
\footnote{Additionally, Parker wrote the following in his abstract for
Yaroslav Sergeyev's conference: ``Recently, theories of infinite
quantity alternative to Cantor's have emerged.  We will consider the
theories of numerosity and grossone''~\cite{Pa19b}.  While the theory
of numerosity by Benci, Di Nasso, and others is a mathematical theory
(see e.g., \cite{Be03}, \cite{Be19}), the situation with Sergeyev's
gross one is summarized at
\url{https://u.math.biu.ac.il/~katzmik/pirahodimetz.html}).}

A methodological bias in favor of purely Archimedean frameworks is not
limited to philosophical analyses of vast pluralities of
trials/tickets.  Thus, the applicability of modern infinitesimal
frameworks to the interpretation of the procedures of the Leibnizian
calculus is an area of lively debate.  Most recently, Arthur and
Rabouin (\cite{Ar24} 2024) sought to challenge the notion that
Leibnizian infinitesimals are non-contradictory, and claim that in
Leibniz, the term \emph{infinitesimal} is non-referring.  In their
latest book \cite{Ar25}, they double down on their interpretation
following Ishiguro \cite[ch.\;5]{Is90}, and claim that Leibnizian
infinitesimals were not \emph{quantities}.  A rebuttal appears in Katz
and Kuhlemann (\cite{23f} 2025).%
\footnote{For further analysis see
\url{https://u.math.biu.ac.il/~katzmik/depictions.html}.}

\section
{Robinson's framework for infinitesimal analysis}
\label{s6}

We recall the basic properties of Robinson's framework of analysis
with infinitesimals. We follow the discussion in \cite[Section 3.1
  and~3.2]{21b}.

\subsection{Constructing hyperreal fields}
\label{f3}

Fields~$\astr$ of hyperreal numbers can be obtained by the so-called
\emph{ultrapower construction}. In this approach, one sets~$\astr =
\R^\N\!/\mathcal{U}$, where~$\mathcal{U}$ is a nonprincipal
ultrafilter over~$\N$. The operations and relations on~$\astr$ are
defined from the quotient structure.  For instance, given~$x = [x_n]$
and~$y=[y_n]$, we set~$x+y=[x_n+y_n]$ and~$x \cdot y=[x_n\cdot y_n]$.
We have~$x < y$ if and only if~$\{n\in\N\colon x_n<y_n\}\in
\mathcal{U}$.  A real number~$r$ is identified with the equivalence
class of the constant sequence~$\langle r \rangle$.

Elements of~$\astr$ can be classified according to their size relative
to elements of~$\R$.  A number~$x \in \astr$ is called
\begin{itemize}
	\item limited  iff there exists~$r \in \R$ such that~$|x|<r$;
	\item unlimited iff it is not limited;
	\item infinitesimal iff for every positive~$r \in \R$,
         ~$|x|<r$.
\end{itemize}

Each limited~$x\in \astr$ is infinitely close to an element of~$\R$,
called the \emph{standard part} of~$x$ and denoted~$\st\, x$.
Similarly, if a subset~$X \subseteq \astr$ contains only limited
elements, one sets~$\st\, X=\{ \st\, x : x \in X \}$.

Let~$\mathcal P=\mathcal P(\R)$ be the power set of~$\R$.  Then the
star transform (associating to every standard entity its nonstandard
analog) produces the entity~${}^{\ast}\mathcal P$.  An \emph{internal}
subset~$A\subseteq\astr$ of~$\astr$ is by definition a member
of\,~${}^{\ast}\mathcal P$.
%
%More concretely, in the ultrapower construction an internal
%subset~$A\subseteq\astr$ is represented by a sequence~$(A_n)$ of
%subsets~$A_n\subseteq\R$.  Here an element~$[x_n]\in\astr$ belongs
%to~$A=[A_n]$ if and only if~$\{n\in\N\colon x_n\in A_n\} \in
%\mathcal{U}$.
%
A subset of~$\astr$ which is not internal is called \emph{external}.%.

More generally, in the context of the star transform from the
superstructure over~$\R$ to the superstructure over~$\astr$, a set~$A$
of the latter is internal if and only if it is a member of
${}^\ast\hskip-2pt Z$ for some~$Z$ in the superstructure over~$\R$.  A
\emph{hyperfinite set} is an internal set admitting an internal 1-1
correspondence with a member of~$\astn$.
\footnote{For further properties of the ultrapower construction of
hyperreal numbers and the superstructures, see e.g., Fletcher et
al.\ \cite{17f} and Goldblatt \cite{goldblatt}.}
\emph{Hyperfinite additivity} of a measure refers to the analog of
finite additivity for an internal hyperfinite family of sets.

Certain applications, such as those of Section~\ref{sec B-track},
require stronger saturation properties than the construction of
hyperreals outlined above; see further in \cite[Theorem 2.2]{bbd}.

\subsection{Axiomatic/syntactic frameworks of Hrbacek and Nelson}
\label{appendix nelson}

The ultrapower approach is not the only way of obtaining fields with
infinitesimals that are powerful enough to develop models of regular
chances; see Section \ref{f3b}.%
\footnote{See also the discussion in \cite[Section 3.4]{21b} and
\cite[Section~4]{21c}.}

In Nelson's Internal Set Theory \cite{nelson}, one works within the
field~$\R$ and finds numbers that behave like infinitesimals there.%
\footnote{A colleague expressed the following (commonly voiced)
objection: ``I do not think that this phrasing: `one works within the
ordinary real line and finds numbers that behave like infinitesimals
there' is correct: an extra predicate and three additional axioms do
change the definition, and the ordinary real line becomes the
(external) set of standard numbers in Nelson’s model of the real
line.''  To formulate a response, it is necessary to analyze the term
\emph{ordinary real line} used by the referee.  Textbooks
traditionally define the ordinary real line $\R$ via Dedekind cuts or
equivalence classes of Cauchy sequences, and $\N$ as the least
inductive set.  Being a conservative extension of ZFC, Nelson's IST
defines $\N$ and $\R$ in exactly the same traditional way.  Thus `an
extra predicate and three additional axioms' \emph{do~not} change the
definition, and when the referee denies that IST's $\R$ is the
ordinary real line, he is evidently not referring to one of the
traditional definitions of $\R$.  What he is apparently referring to
is an entity described as a \emph{standard model} a.k.a. an
\emph{intended interpretation} (which is assumed not to contain
certain numbers expressible in IST).  We note that the existence of
such an entity is a realist philosophical assumption rather than a
mathematical fact.  Robinson specifically rejected such an assumption
in \cite[p.\;242]{Ro65}.}
This is possible by enriching the language of set theory through the
introduction of a one-place predicate~$\st$ called `standard',
together with additional axioms governing its interaction with the
axioms of Zermelo–Fraenkel set theory.  To summarize, Nelson's theory
is formulated in the~$\st$–$\in$–language, whereas traditional set
theories such as ZFC are formulated in the
$\in$–language.%
\footnote{Nelson applied these ideas to probability theory in
  \cite{Ne87}.}
At the same time, Karel Hrbacek \cite{hrbacek} developed a different
axiomatic framework, now called HST.%
\footnote{Hrbacek's theory is formulated in the~$\mathfrak
S$–$\mathfrak I$–$\in$–language, where~$\mathfrak S$ is a unary
predicate analogue to~$\st$ and~$\mathfrak I$ is an unary predicate
that express internality.  For further details, see Fletcher et
al.~\cite{17f}.}

In Nelson's and Hrbacek's frameworks, unlimited and infinitesimal
numbers are found in the field~$\R$ itself rather than in its proper
extension.  For instance, an infinitesimal~$\varepsilon$ is a real
number satisfying~$|\varepsilon| < r$ for every standard
positive~$r\in\R$.  The possibility of viewing infinitesimals as being
found within~$\R$ itself undermines philosophical criticisms anchored
in the extension view, such as Parker's claim that infinitesimal
probabilities depend on `misleading excess structure'.%
\footnote{Parker \cite[p.\;9]{Pa21}.  See further in
Section~\ref{s55}.  Similar objections voiced by Easwaran and Towsner
have been addressed in \cite[Section~2.8]{19c}.}
Axiomatic frameworks have sometimes been criticized on the grounds
that they are allegedly unable to handle certain hyperreal
constructions, such as those of nonstandard hulls and Loeb measures.
Hrbacek and Katz \cite{23c} show that axiomatic frameworks can indeed
handle such constructions.

Axiomatic frameworks SPOT and SCOT for NSA that are conservative
respectively over ZF and ZF+ADC (here ADC is the axiom of dependent
choice) were developed in \cite{21e}.

\subsection{The transfer principle of Robinson's framework}
\label{appendix a2}

The transfer principle asserts that the {internal objects} of
Robinson's framework or all objects in Nelson's framework 
%\footnote{Nelson explicitly states that ``external sets are not
%entities of IST'' \cite{nelson}.}
satisfy all the first-order properties of the corresponding classical
objects.

Kanovei et al.\ describe the transfer principle as
\begin{enumerate}\item[]
	a type of theorem that, depending on the context, asserts that
	rules, laws or procedures valid for a certain number system,
	still apply (i.e., are ``transferred") to an extended number
	system.%
        \footnote{Kanovei et al.~\cite[p.\;113]{18i}.}
\end{enumerate}
For the role of the transfer principle in a mathematical theory with
infinitesimals, we refer to \cite[Sections 3.3 and 3.4]{21b}.

\section{Conclusion}
\label{s7}

We have analyzed the extent to which Archimedean and non-Archi\-medean
models of a fair spinner satisfy the seemingly natural and exhaustive
hypotheses, namely: symmetry under real rotations {\sy}$_{\R}$,
uniformity~{\Un}, regularity \textbf{Reg}, and {\Po} (each point
outcome is possible).  Currently, no known model is able to satisfy
all these hypotheses simultaneously.  Thus Archimedean models reject
regularity, while non-Archi\-medean models weaken {\sy}$_{\R}$.

Regular models of the fair spinner have been criticized on the grounds
that they allegedly do not satisfy physical intuitions and introduce
undesirable extrinsic features, being non-unique.  An examination of
the existing models of the spinner indicates that such a critique
applies to Archimedean models, as well.  Indeed, neither the
hypotheses \mbox{${\sy}_{\R}+{\Un}+{\Po}$} (or even the stronger
${\sy}_{\R}+{\Un}+\textbf{Tot}$, where \textbf{Tot} denotes the
hypothesis that every outcome is possible) nor the
hypotheses~\mbox{${\Un}+\textbf{Tot} + \textbf{Reg}$} are sufficient
to uniquely determine a model, unless one renounces strong choice
principles that are necessary for proving fundamental theorems of
set-theoretic mathematics.  Since there are multiple Archimedean
models of the fair spinner, non-Archimedean models cannot be ruled out
simply on the grounds that they are non-unique, as Archimedean models
are similarly non-unique.

It is possible to recover uniqueness by focusing on a minimal model.
In such case, we saw that the minimal Archimedean model
%that assigns a probability to every point outcome
can be regularized to a non-Archimedean model with the additional
property \textbf{Reg} as well as the restriction of the property
{\sy}$_{\R}$ to the algebra generated by half-open intervals.

The non-Archimedean models can, moreover, be further sharpened with
additional hypotheses, such as {\Un}$_{\mathfrak B}$, namely
uniformity over an algebra of Lebesgue measurable sets where every
nonempty set has a positive measure, or \textbf{Co}$_{H}$, namely
coherence with the Hausdorff~$t$-measures for~$t\in[0,+\infty)$, that
improve the translation invariance or add coherence with a family of
real-valued measures.  These additional features may be relevant in
various applications, such as the study of Brownian motion or analysis
on fractals \cite{lindstrom, lindstrom 2}.  The most recent work in
probability and statistics using NSA methods includes Duanmu et
al.~(\cite{Du23} 2023), Alam and Sengupta (\cite{Al23} 2023), Anderson
et al.~(\cite{An24}~2024).

\section*{Acknowledgments}

We are grateful to Karel Hrbacek and the anonymous referee for helpful
comments.


\begin{thebibliography}{12}
  

\bibitem{Al23} Alam, Irfan; Sengupta, Ambar N.\, Large-$N$ limits of
  spaces and structures.  \emph{Commun. Pure Appl. Anal}.  \textbf{22}
  (2023), no.\;4, 1009--1042.  \MR{4583546} \ZBL{07680182}


\bibitem{An24} Anderson, Robert M.; Duanmu, Haosui; Smith, Aaron.\,
  Mixing times and hitting times for general Markov processes.
  \emph{Israel J. Math}.  \textbf{259} (2024), no.\;2, 759--834.
  \MR{4732980} \ZBL{1537.60094}

\bibitem{Ar24} Arthur, Richard; Rabouin, David.\, On the unviability
  of interpreting Leibniz's infinitesimals through non-standard
  analysis.  \emph{Historia Math}. \textbf{66} (2024), 26--42.
  \MR{4721554} \ZBL{07837727}


\bibitem{Ar25} Arthur, Richard; Rabouin, David.  \emph{Leibniz on the
Foundations of the Differential Calculus}.  Springer, Cham, 2025.


\bibitem{barrett10} Barrett, Martin.\, The possibility of
  infinitesimal chances.\, In \emph{The Place of Probability in
  Science}, Springer, Dordrecht, 2010, pp.\;65--79.
%No MR, no Zb


	
\bibitem{bbd} Benci, Vieri; Bottazzi, Emanuele; Di Nasso, Mauro.
  Elementary numerosity and measures. \emph{Journal of Logic and
  Analysis} \textbf{6} (2014), Paper 3, 14 pp.  \MR{3257035}
  \ZBL{1300.26018}


\bibitem{bbd2} Benci, Vieri; Bottazzi, Emanuele; Di Nasso, Mauro.
  Some applications of numerosities in measure theory.
  \emph{Rendiconti Lincei-Matematica E Applicazioni} \textbf{26}
  (2015), no.\,1, 37--47.  \MR{3345320} \ZBL{1309.26028}

  
\bibitem{Be03} Benci, Vieri; Di Nasso, Mauro.\, Numerosities of
  labelled sets: a new way of counting.  \emph{Adv. Math}.
  \textbf{173} (2003), no.\,1, 50--67.  \MR{1954455}
  \ZBL{1028.03042}


\bibitem{Be19} Benci, Vieri; Di Nasso, Mauro.\, \emph{How to measure
the infinite.  Mathematics with infinite and infinitesimal numbers}.
  World Scientific, Hackensack, NJ, 2019.  \MR{3889056}
  \ZBL{1429.26001}


\bibitem{benci2013} Benci, Vieri; Horsten, Leon; Wenmackers, Sylvia.\,
  Non-Archimedean probability.  \emph{Milan Journal of Mathematics}
  \textbf{81} (2013), 121--151.  \MR{3046984} \ZBL{1411.60007}

	
\bibitem{benci2018} Benci, Vieri; Horsten, Leon; Wenmackers, Sylvia.\,
  Infinitesimal probabilities.  \emph{British Journal for the
  Philosophy of Science} \textbf{69} (2018), 509--552.  \MR{3805192}
  \ZBL{1400.03006}

	
\bibitem{bernstein+wattenberg} Bernstein, Allen; Wattenberg, Frank.\,
  Non-standard measure theory, in \emph{Applications of model theory
  to algebra, analysis, and probability}, edited by W.A.J. Luxemburg
  (Holt, Rinehart, and Winston, New York, 1969), pp.\,171--185.
  \MR{0247018} \ZBL{0195.01404}
	

\bibitem{Bo21} Bottazzi, Emanuele; Eskew, Monroe.\, Integration with
  filters.  \emph{Journal of Logic and Analysis} \textbf{14} (2022),
  Paper No.\,1, 54 pp.  \MR{4488966} \ZBL{1511.28013}


\bibitem{19c} Bottazzi, Emanuele; Kanovei, Vladimir; Katz, Mikhail;
  Mormann, Thomas; Sherry, David.\, On mathematical realism and the
  applicability of hyperreals.  \emph{Mat. Stud}.  \textbf{51} (2019),
  no.\;2, 200--224.  \MR{3988243} \ZBL{1436.00032}
%
%realism
	

\bibitem{21b} Bottazzi, Emanuele; Katz, Mikhail.\, Infinite lotteries,
  spinners, applicability of hyperreals.  \emph{Philosophia
  Mathematica} \textbf{29} (2021), no.\,1, 88--109.  \MR{4267988}
  \ZBL{1505.03145}
  
	
\bibitem{21c} Bottazzi, Emanuele; Katz, Mikhail.\, Internality,
  transfer, and infinitesimal modeling of infinite processes.
  \emph{Philosophia Mathematica} \textbf{29} (2021), no.\;2, 256--277.
  \MR{4492449} \ZBL{1505.03146}
	
  
\bibitem{failing} Bridges, Douglas; McKubre-Jordens, Maarten.\,
  Solving the Dirichlet problem constructively.  \emph{Journal of
  Logic and Analysis} \textbf{5} (2013), Paper 3, 22 pp.  \MR{3033506}
  \ZBL{1296.03034}


\bibitem{Ca63} Carnap, Rudolf.\, Carnap's Intellectual Autobiography.
  In \emph{The Philosophy of Rudolf Carnap}, The Library of Living
  Philosophers, Vol. XI, ed. Paul Arthur Schilpp, 1963.
%No MR, no Zb


\bibitem{Ca95} Cartwright, Nancy; Shomar, Towfic; Su\'arez,
  Mauricio.\, The tool box of science.  Tools for the building of
  models with a superconductivity example.  \emph{Poznan Studies in
  the Philosophy of the Sciences and the Humanities} \textbf{44}
  (1995), 137--149.
%No MR, no Zb


\bibitem{howgood} Ciesielski, Krzysztof.  How good is Lebesgue
  measure?. \emph{The Mathematical Intelligencer} \textbf{11} (1989),
  54--58.  \MR{0994965} \ZBL{0679.28001}



\bibitem{Ci90} Ciesielski, Krzysztof.\, Isometrically invariant
  extensions of Lebesgue measure.  \emph{Proc. Amer. Math. Soc}.
  \textbf{110} (1990), no.\;3, 799--801.  \MR{1027089}
  \ZBL{0707.28008}


  \bibitem{extensions} Ciesielski, Krzysztof; Pelc, Andrzej.\,
    Extensions of invariant measures on Euclidean
    spaces. \emph{Fundamenta Mathematicae} \textbf{125} (1985),
    no.\,1, 1--10.  \MR{0813984} \ZBL{0593.28011}
  

\bibitem{Co46} Cox, R. T.\, Probability, frequency and reasonable
  expectation.  \emph{Amer. J. Phys}.  \textbf{14} (1946), 1--13.
  \MR{0015688} \ZBL{0063.01001}

	
\bibitem{Di23} Dinis, Bruno.\, Equality and near-equality in a
  nonstandard world.  \emph{Log. Log. Philos}.  \textbf{32} (2023),
  no.\,1, 105--118.  \MR{4562928} \ZBL{1530.03025}


\bibitem{Du23} Duanmu, Haosui; Schrittesser, David; Weiss, William.\,
  Loeb extension and Loeb equivalence II.  \emph{Fund. Math}.
  \textbf{262} (2023), no.\,1, 71--83.  \MR{4618646} \ZBL{07797034}


\bibitem{Ea14} Easwaran, Kenny.\, Regularity and hyperreal credences.
  \emph{Philosophical Review} \textbf{123} (2014), no.\,1, 1--41.
%No MR, no Zb
  

\bibitem{et} Easwaran, Kenny; Towsner, Henry.\, Realism in
  mathematics: The case of the hyperreals.
  \url{https://www.dropbox.com/s/vmgevvgmy9bhl2c/Hyperreals.pdf?raw=1}
  %No MR, no Zb
  
		
\bibitem{Ed63} Edwards, W.; Lindman, H.; Savage, L. J.\, Bayesian
  Statistical Inference for Psychological Research.
  \emph{Psychological Review} \textbf{70} (1963), 193--242.
  \MR{0785366} \ZBL{0173.22004}


\bibitem{El04} Elga, Adam.\, Infinitesimal chances and the laws of
  nature.  \emph{Australasian Journal of Philosophy} \textbf{82}
  (2004), 67--76.
%No MR, no Zb


\bibitem{17f} Fletcher, Peter; Hrbacek, Karel; Kanovei, Vladimir;
  Katz, Mikhail; Lobry, Claude; Sanders, Sam.\, Approaches to analysis
  with infinitesimals following Robinson, Nelson, and others.
  \emph{Real Analysis Exchange} \textbf{42} (2017), no.\;2, 193--252.
  \MR{3721800} \ZBL{1435.03094}


\bibitem{Fo91} Foreman, Matthew; Wehrung, Friedrich.\, The Hahn-Banach
  theorem implies the existence of a non-Lebesgue measurable set.
  \emph{Fundamenta Mathematicae} \textbf{138} (1991), no.\;1, 13--19.
  \MR{1122273} \ZBL{0792.28005}
  

\bibitem{fremlin} Fremlin, David.  \emph{Measure theory}.  Vol.\;4.I.
  Torres Fremlin, Colchester, 2006.  \MR{2462372} \ZBL{1166.28001}


\bibitem{Fr23} Frigg, Roman.  \emph{Models and Theories.  A
Philosophical Inquiry}.  Routledge Taylor \& Francis, London and New
  York, 2023.
%No MR, no Zb


	
\bibitem{goldblatt} Goldblatt, Robert.  \emph{Lectures on the
hyperreals.  An introduction to nonstandard analysis}.  Graduate Texts
  in Mathematics, 188.  Springer-Verlag, New York, 1998.  \MR{1643950}
  \ZBL{0911.03032}

	
\bibitem{staying} H\'ajek, Alan.\, Staying Regular?.  In preparation,
  20.7.2021 version.  See
  \url{http://hplms.berkeley.edu/HajekStayingRegular.pdf}
% Hajek No MR, no Bl


\bibitem{Ha77} Harazi\v{s}vili, A. B.\, On Sierpi\'nski's problem of
  strict extendability of an invariant measure.
  \emph{Dokl. Akad. Nauk SSSR} \textbf{232} (1977), no.\;2, 288--291.
  \MR{0584928} \ZBL{0403.28015}


\bibitem{Ho98} Howard, Paul; Rubin, Jean E.  \emph{Consequences of the
axiom of choice}.  Mathematical Surveys and Monographs, 59.  American
  Mathematical Society, Providence, RI, 1998.  \MR{1637107}
  \ZBL{0947.03001}

	
\bibitem{hrbacek} Hrbacek, Karel.\, Axiomatic foundations for
  nonstandard analysis.  \emph{Fundamenta Mathematicae} \textbf{98}
  (1978), no.\,1, 1–19.  \MR{0528351} \ZBL{0373.02039}
  
	
\bibitem{21e} Hrbacek, Karel; Katz, Mikhail.\, Infinitesimal analysis
  without the Axiom of Choice.  \emph{Annals of Pure and Applied
  Logic} \textbf{172} (2021), no.\;6, paper no.~102959, 31 pp.
  \MR{4224071} \ZBL{1529.03285}

\bibitem{23c} Hrbacek, Karel; Katz, Mikhail.\, Constructing
  nonstandard hulls and Loeb measures in internal set theories.
  \emph{Bulletin of Symbolic Logic} \textbf{29} (2023), no.\,1,
  97--127.  \MR{4560535} \ZBL{1512.03077}
  
  
\bibitem{23d} Hrbacek, Karel; Katz, Mikhail.\, Effective
  infinitesimals in~$\R$.  \emph{Real Analysis Exchange} \textbf{48}
  (2023), no.\;2, 365--380.  \MR{4668954} \ZBL{1538.03037}


\bibitem{23e} Hrbacek, Karel; Katz, Mikhail.\, Peano and Osgood
  theorems via effective infinitesimals.  \emph{Journal of Logic and
  Analysis} \textbf{15}:6 (2023), 1--19.  \MR{4673816} %No Zb


\bibitem{Is22} Isaacs, Y.; H\'ajek, A.; Hawthorne, J.\,
  Non-measurability, imprecise credences, and imprecise chances.
  \emph{Mind} \textbf{131} (2022), no.\;523, 892--916.  \MR{4504885}
%No Zb Hajek


\bibitem{Is90} Ishiguro, H.  \emph{Leibniz's philosophy of logic and
  language.}  Second edition.  Cambridge: Cambridge University Press,
  1990.



\bibitem{18i} Kanovei, Vladimir; Katz, Karin; Katz, Mikhail; Mormann,
  Thomas.\, What makes a theory of infinitesimals useful?  A view by
  Klein and Fraenkel.  \emph{Journal of Humanistic Mathematics}
  \textbf{8} (2018), no.\;1, 108--119.  \MR{3762866} \ZBL{1522.03352}
	

\bibitem{17i} Kanovei, Vladimir; Katz, Mikhail.\, A positive function
  with vanishing Lebesgue integral in Zermelo--Fraenkel set theory.
  \emph{Real Analysis Exchange} \textbf{42}(2), 2017, 385--390.
  \MR{3721807} \ZBL{1423.03180}
	

%\bibitem{23z} Kanovei, V.; Katz, M.; Kudryk, T.; Kuhlemann, K.  
%A philosophical history of infinitesimals.
%
%phi.tex
%


\bibitem{23f} Katz, Mikhail; Kuhlemann, Karl.\, Leibniz's contested
  infinitesimals: Further depictions.  \emph{Ga\d{n}ita Bhāratī}
  \textbf{45} (2023), no.\,1, 1--36.
  \url{https://doi.org/10.32381/GB.2022.45.1.4},
  \url{https://arxiv.org/abs/2501.01193} (to appear in 2025)
%No MR or Zb yet



\bibitem{Ke55} Kemeny, John.\, Fair bets and inductive probabilities.
  \emph{Journal of Symbolic Logic} \textbf{20} (1955), no.\;3,
  263--273.  \MR{0074699} \ZBL{0066.11002}


\bibitem{lindstrom} Lindstr{\o}m, Tom.\, \emph{Brownian motion on
nested fractals}.  \emph{Memoirs of the American Mathematical
Society}, No.\;420.  American Mathematical Soc., 1990.  \MR{0988082}
  \ZBL{0688.60065}
  
        
\bibitem{lindstrom 2} Lindstr{\o}m, Tom.  Brownian motion penetrating
  the Sierpinski gasket.  In K.~D.~Elworthy and N. Ikeda (eds),
  \emph{Asymptotic Problems in Probability Theory: Stochastic Models
  and Diffusions on Fractals}, Longman, 1993, pp.\;248--278.
  \MR{1354158} \ZBL{0787.60101}


\bibitem{Ma24} Mancosu, Paolo; Massas, Guillaume.\, Totality,
  regularity, and cardinality in probability theory.  \emph{Philosophy
  of Science} \textbf{91} (2024), 721--740.  \MR{4796641} 
  %No Zb
  
	
\bibitem{nelson} Nelson, Edward.\, Internal set theory: a new approach
  to nonstandard analysis.  \emph{Bulletin of the American
  Mathematical Society} \textbf{83} (1977), no.\;6, 1165--1198.
  \MR{0469763} \ZBL{0373.02040}


\bibitem{Ne87} Nelson, Edward.\, \emph{Radically Elementary
Probability Theory}.  Annals of Math. Studies 117, Princeton
  University Press, Princeton, NJ, 1987.  \MR{0906454} \ZBL{0651.60001}


\bibitem{No18a} Norton, John D.\, How to build an infinite lottery
  machine.  \emph{European Journal for Philosophy of Science}
  \textbf{8} (2018), no.\,1, 71--95.  \MR{3761814} \ZBL{1398.60004}


\bibitem{No22} Norton, John D.; Parker, Matthew W.\, An infinite
  lottery paradox.  \emph{Axiomathes} \textbf{32} (2022), 1--6.
%No MR, no Zb


\bibitem{No18b} Norton, John D.; Pruss, Alexander R.\, Correction to
  John D. Norton: ``How to build an infinite lottery machine.''
  \emph{European Journal for Philosophy of Science} \textbf{8} (2018),
  no.\,1, 143--144.  \MR{3761818} \ZBL{1398.60006}

	
\bibitem{Pa13} Parker, Matthew W.\, Set size and the part-whole
  principle.  \emph{Review of Symbolic Logic} \textbf{6} (2013),
  no.\;4, 589--612.  \MR{3182257} \ZBL{1328.03047}


\bibitem{Pa19} Parker, Matthew W.\, Symmetry arguments against regular
  probability: a reply to recent objections.  \emph{European Journal
  for Philosophy of Science} \textbf{9} (2019), no.\;1, Art.\;8,
  pp.\,1--21.  \MR{3873736} \ZBL{1428.60004}
  

\bibitem{Pa19b} Parker, Matthew W.\, Applied Infinity: Accuracy and
  Artefact.  In Book of abstracts of The 3rd International Conference
  and Summer School Numerical Computations: Theory and Algorithms
  NUMTA 2019.  Edited by Yaroslav D. Sergeyev, Dmitri E. Kvasov, Marat
  S. Mukhametzhanov, Maria Chiara Nasso.  University of Calabria,
  Rende (CS), Italy, 2019, p.\,148.
%No MR, no Zb


\bibitem{Pa20} Parker, Matthew W.\, Comparative infinite lottery
  logic.  \emph{Studies in History and Philosophy of Science}
  \textbf{84} (2020) 28--36.
%No MR, no Zb


\bibitem{Pa21} Parker, Matthew W.\, Weintraub's response to
  Williamson's coin flip argument.  \emph{European Journal for
  Philosophy of Science} \textbf{11} (2021), article 71, 21~pp.
  \MR{4292692}
%No Zb


\bibitem{Pr13} Pruss, Alexander R.  Null probability, dominance and
  rotation.  \emph{Analysis (Oxford)} \textbf{73} (2013), no.\;4,
  682--685.  \MR{3118841} \ZBL{1329.60006}



\bibitem{Pr14} Pruss, Alexander R.\, Regular probability comparisons
  imply the Banach--Tarski Paradox.  \emph{Synthese} \textbf{191}
  (2014), 3525--3540.  \MR{3254636} \ZBL{1310.03008}
  
	
\bibitem{Pr21a} Pruss, Alexander R.\, Underdetermination of
  infinitesimal probabilities.  \emph{Synthese} \textbf{198} (2021),
  777--799.  \MR{4204815} \ZBL{1506.62216}


\bibitem{Pr21b} Pruss, Alexander R.\, Non-classical probabilities
  invariant under symmetries.  \emph{Synthese} \textbf{199} (2021),
  no.\;3--4, 8507--8532.  \MR{4351364} \ZBL{1529.60009}
	

\bibitem{Pr22} Pruss, Alexander R.\, Correction: Non-classical
  probabilities invariant under symmetries.  \emph{Synthese}
  \textbf{200} (2022), no.\;5, Paper No.\;385, 4 pp.  \MR{4480339}
  %No Zb
  

\bibitem{Ro65} Robinson, Abraham.\, Formalism 64.  In \emph{Logic,
Methodology and Philos. Sci. (Proc. 1964 Internat. Congr.)},
  pp.\;228--246, North-Holland, Amsterdam, 1965.  \MR{0214431}
  \ZBL{0199.00204}

\bibitem{Ro66} Robinson, Abraham.\, \emph{Non-standard analysis}.
  North-Holland Publishing, Amsterdam, 1966.  \MR{0205854}
  \ZBL{0151.00803}

	
\bibitem{royden} Royden, Halsey Lawrence.  \emph{Real Analysis.  Third
Edition}.
%Vol. 32.  
New York: Macmillan, 1988.  \MR{1013117} \ZBL{0704.26006}
	

\bibitem{Sa77} Sachs, Rainer Kurt; Wu, Hung Hsi.  \emph{General
relativity for mathematicians}.  Graduate Texts in Mathematics,
  Vol. 48.  Springer-Verlag, New York-Heidelberg, 1977.  \MR{0503498}
  \ZBL{0373.53001}


\bibitem{Sa69} Sacks, Gerald E.\, Measure-theoretic uniformity in
  recursion theory and set theory.  \emph{Trans. Amer. Math. Soc}.
  \textbf{142} (1969), 381--420.  \MR{0253895} \ZBL{0209.01603}


\bibitem{Sh55} Shimony, Abner.\, Coherence and the axioms of
  confirmation.  \emph{Journal of Symbolic Logic} \textbf{20} (1955),
  128.  \MR{0069784} \ZBL{0064.24402}


\bibitem{Sh70} Shimony, Abner.\, Scientific Inference.  In Robert
  Colodny (ed.), \emph{The Nature and Function of Scientific
    Theories}, University of Pittsburgh Press, 1970.
%No MR, no Zb


\bibitem{simpson} Simpson, Stephen.\, \emph{Subsystems of Second Order
Arithmetic, second ed.}.  Cambridge University Press, New York, 2009,
  444 pp.  \MR{2517689} \ZBL{1181.03001}
	

\bibitem{So70} Solovay, Robert.\, A model of set-theory in which every
  set of reals is Lebesgue measurable.  \emph{Annals of Mathematics
  (2)} \textbf{92} (1970), 1--56.  \MR{0265151} \ZBL{0207.00905}
	

\bibitem{spohn} Spohn, Wolfgang.\, Chance and Necessity: From Humean
  Supervenience to Humean Projection.  In \emph{The place of
  probability in science}, Springer, Dordrecht, 2010, pp.\;65--79.
%No MR
  \ZBL{1218.60002}
  

\bibitem{St70} Stalnaker, Robert C.\, Probability and Conditionals.
  \emph{Philosophy of Science} \textbf{37} (1970), 64--80.
  \MR{0285345}
%No Zb


\bibitem{Ca08} Su\'arez, Mauricio; Cartwright, Nancy.  \emph{Theories:
  Tools versus models}.  Studies in History and Philosophy of Science
  Part B: Studies in History and Philosophy of Modern Physics,
  Vol. 39, Issue 1 (2008), 62--81.
%No MR
  \ZBL{1282.82054}
  
	
\bibitem{watt} Wattenberg, Frank.\, Nonstandard measure theory.
  Hausdorff measure.  \emph{Proceedings of the American Mathematical
  Society} \textbf{65} (1977), no.\;2, 326--331.  \MR{0444466}
  \ZBL{0365.28016}
	

\bibitem{weintraub} Weintraub, Ruth.\, How probable is an infinite
  sequence of heads?  A reply to Williamson.  \emph{Analysis}
  \textbf{68} (2008), 247--250.  \MR{2426253}
%No Zb


\bibitem{We19} Wenmackers, Sylvia.\, Infinitesimal probabilities.  In
  R. Pettigrew and J. Weisberg (eds.)  \emph{The Open Handbook of
    Formal Epistemology}, PhilPapers Foundation, pp.\,199--265 (2016).
%  \url{https://philpapers.org/rec/WENIP}
%in particular pp. 244-246
%
%Petttigrew volume saved in pettigrew19.pdf
%
%No MR, no Zb


\bibitem{Wi07} Williamson, Timothy.\, How probable is an infinite
  sequence of heads?  \emph{Analysis} \textbf{67} (2007), 173--180.
  \MR{2339283} \ZBL{1158.60304}
  

\end{thebibliography}
\end{document}